\documentclass[12pt]{amsart}

\usepackage{amssymb}
\usepackage[cmtip,all]{xy}
\usepackage{verbatim}
\usepackage{upref}
\usepackage{url}
\usepackage[normalem]{ulem}

\allowdisplaybreaks

\newtheorem{thm}{Theorem}[section]
\newtheorem{lem}[thm]{Lemma}
\newtheorem{cor}[thm]{Corollary}
\newtheorem{prop}[thm]{Proposition}

\theoremstyle{definition} 
\newtheorem{defn}[thm]{Definition}
\newtheorem{ex}[thm]{Example}
\newtheorem{q}[thm]{Question}
\newtheorem{rem}[thm]{Remark}

\numberwithin{equation}{section}
\renewcommand{\theenumi}{\roman{enumi}}

\allowdisplaybreaks
\newcommand{\secref}[1]{Section~\textup{\ref{#1}}}
\newcommand{\subsecref}[1]{Subsection~\textup{\ref{#1}}}
\newcommand{\thmref}[1]{Theorem~\textup{\ref{#1}}}
\newcommand{\corref}[1]{Corollary~\textup{\ref{#1}}}
\newcommand{\lemref}[1]{Lemma~\textup{\ref{#1}}}
\newcommand{\propref}[1]{Proposition~\textup{\ref{#1}}}

\newcommand{\defnref}[1]{Definition~\textup{\ref{#1}}}
\newcommand{\remref}[1]{Remark~\textup{\ref{#1}}}

\newcommand{\exref}[1]{Example~\textup{\ref{#1}}}

\newcommand{\KK}{\mathcal K}

\newcommand{\LL}{\mathcal L}
\newcommand{\OO}{\mathcal O}
\newcommand{\TT}{\mathcal T}

\newcommand{\N}{\mathbb N}
\newcommand{\Z}{\mathbb Z}
\newcommand{\C}{\mathbb C}
\newcommand{\T}{\mathbb T}

\newcommand{\then}{\ensuremath{\Rightarrow}}
\renewcommand{\iff}{\ensuremath{\Leftrightarrow}}
\renewcommand{\bar}{\overline}

\newcommand{\wilde}{\widetilde}
\newcommand{\inv}{^{-1}}
\newcommand{\<}{\langle}
\renewcommand{\>}{\rangle}

\newcommand{\ann}{^\perp}

\newcommand{\Chi}{\raisebox{2pt}{\ensuremath{\chi}}}
\renewcommand{\epsilon}{\varepsilon}
\newcommand{\case}{& \text{if }}
\newcommand{\ifnot}{& \text{otherwise}}
\newcommand{\minus}{\setminus}
\newcommand{\act}{\curvearrowright}
\renewcommand{\:}{\colon}

\renewcommand{\)}{\textup)}

\newcommand{\id}{\text{\textup{id}}}

\DeclareMathOperator{\ad}{Ad}
\DeclareMathOperator*{\spn}{span}
\DeclareMathOperator*{\clspn}{\overline{\spn}}

\newcommand{\eplist}{\renewcommand{\labelenumi}{(\textbf{\eplabel} \theenumi)}}
\newcommand{\eplabel}{\textbf{EP}}
\newcommand{\ep}[1]{\textup{(\eplabel\ {\ref{#1}})}}

\newcommand{\epprimelist}{\renewcommand{\labelenumi}{(\textbf{\epprimelabel} \theenumi)}}
\newcommand{\epprimelabel}{\textbf{EP'}}

\newcommand{\midtext}[1]{\quad\text{#1}\quad}
\newcommand{\righttext}[1]{\quad\text{#1 }}

\newcommand{\yphi}{Y^\varphi}
\newcommand{\ty}{\TT_{\yphi}}
\newcommand{\oy}{\OO_{\yphi}}
\newcommand{\rg}{\textup{rg}}
\newcommand{\so}{\textup{so}}

\begin{document}
\title{On Exel-Pardo algebras}
\author{Erik B\'edos}
\address{Institute of Mathematics, University of Oslo, PB 1053 Blindern, 0316 Oslo, Norway}
\email{bedos@math.uio.no}
\author{S. Kaliszewski}
\address{School of Mathematical and Statistical Sciences, Arizona State University, Tempe, AZ 85287}
\email{kaliszewski@asu.edu}
\author{John Quigg}
\address{School of Mathematical and Statistical Sciences, Arizona State University, Tempe, AZ 85287}
\email{quigg@asu.edu}
\date{October 27, 2017}


\begin{abstract}
We generalize a recent construction of Exel and Pardo,
from discrete groups acting on finite directed graphs
to locally compact groups acting on topological graphs.
To each cocycle for such an action, we construct a $C^*$-correspondence
whose associated Cuntz-Pimsner algebra is the analog of the Exel-Pardo $C^*$-algebra.
\end{abstract}

\maketitle

\section{Introduction}
\label{intro}

Let $E$ denote a directed graph with vertex set $E^0$ and edge set $E^1$,
as in \cite[Section 5]{Rae}.
When $E$ is finite with no sources, Exel and Pardo 
have shown in \cite{EP} how to attach a Cuntz-Krieger-like $C^*$-algebra 
$\OO_{G,E}$
to an action 
of a countable discrete group $G$ on $E$, equipped with a cocycle $\varphi$ from $ G \times E^1$ into $G$ that is compatible with the action of $G$ on $E^0$. This set-up is powerful as it encompasses the $C^*$-algebras $\OO_{(G,X)}$ of self-similar groups 
introduced by Nekrashevych \cite{Nek09}
and the Katsura's algebras $\OO_{A,B}$ associated with two $N\times N$ integer matrices $A$ and $B$ in \cite{Kat08}. Our aim with this paper is to generalize Exel and Pardo's construction, allowing $G$ to be uncountable and $E$ to be infinite, possibly with sources.
In fact, we develop the basic construction of the $C^*$-algebra at a significantly greater level of generality: we start with a locally compact group $G$ acting on a topological graph $E$ and a  cocycle $\varphi$ for this action.

The main idea is that there is a natural way to associate to the given data $(E, G, \varphi)$ a $C^*$-correspondence $\yphi$ 
over the crossed product $C_0(E^0) \rtimes G$, and the Cuntz-Pimsner algebra associated with this correspondence provides the desired algebra. In \cite[Section 10]{EP}, Exel and Pardo also give a description of $\OO_{G,E}$ (in the case they consider) as a Cuntz-Pimsner algebra, but our approach has an interesting conceptual feature, besides that it works without any restriction on $G$ and $E$. Considering first the case where $\varphi$ is the ``trivial'' cocycle, that sends $(g, e)$ to $g$ for every $g\in G,  e \in E^1$, 
our correspondence $\yphi$ reduces to 
the crossed product $X\rtimes G$ of the graph correspondence $X=X_E$ by the action of $G$ 
 naturally associated with the action of $G$ on $E$. 
For a general $\varphi$, the correspondence $\yphi$ is equal to $X\rtimes G$ as a right Hilbert $(C_0(E_0)\rtimes G)$-module, but the cocycle $\varphi$ is used to deform the left action of $C_0(E_0)\rtimes G$ on $X\rtimes G$. 

The paper is organized as follows. In Section 2 we review some facts about cocycles for actions of locally compact groups on locally compact Hausdorff spac\-es. In Section 3 we consider an action of a locally compact group $G$ on a topological graph $E$ as in \cite[Section~3]{dkq}, introduce the concept of a cocycle $\varphi$ for such an action (in a slightly more general way than Exel and Pardo) and show how to construct the desired $C^*$-correspondence $Y^\varphi$. As a result, we can form the Toeplitz algebra $\TT_{\yphi}$ and the Cuntz-Pimsner algebra $\OO_{\yphi}$ associated to $\yphi$. In Section 4 we show that if two systems $(E, G, \varphi)$  and $(E',G, \varphi')$  are cohomology conjugate in a natural sense, then $Y^\varphi$ is isomorphic to $Y^{\varphi'}$, hence the resulting algebras are isomorphic. In Section 5 we restrict our attention to the case where $G$ is discrete and $E$ is a directed graph and give a description of $\TT_{Y^\varphi}$ in terms of generators and relations. When $E$ is row-finite, we also give a similar description of $\OO_{\yphi}$, which in particular shows that $\OO_{\yphi}$ is isomorphic to the Exel-Pardo algebra $\OO_{G,E}$ when $E$ is finite and sourceless. 
 For completeness we also show in Section 6 that the Exel-Pardo correspondence obtained in \cite[Section 10]{EP}  is isomorphic to our $Y^\varphi$. Finally, in Section 7 we present several examples of triples $(E, G, \varphi)$ that illustrate the flexibility of our setting and indicate the diversity of $C^*$-algebras that arise from this construction.

\section{Preliminaries}\label{prelims}

We recall some of the well-known theory of cocycles for group actions
(see, e.g., \cite[Section~4.2]{Zimmer}).
Let $G$ be a locally compact group acting continuously by homeomorphisms on a
nonempty
locally compact Hausdorff space $S$,
and let $T$ be a locally compact group.
We sometimes write the action as a map $\sigma\:G\times S\to S$,
and we also write
\[
gx=g\cdot x=\sigma(g,x)\righttext{for}g\in G,x\in S.
\]
A \emph{cocycle for the action $G\act S$ with values in $T$}
is a continuous map $\varphi\:G\times S\to T$ satisfying the \emph{cocycle identity}
\[
\varphi(gh,x)=\varphi(g,hx)\varphi(h,x)\righttext{for all}g,h\in G,x\in S.
\]
We will primarily be concerned with the case $T=G$.

When $S$ is discrete,
the action $G\act S$ is a disjoint union of transitive actions on the orbits $Gx$,
and the restriction $\varphi|_{G\times Gx}$ is a cocycle $\varphi_x$ for this transitive action.
In fact, the cocycle $\varphi$ can be reconstructed from these restricted cocycles $\varphi_x$; indeed the cocycles for the actions on the orbits may be chosen willy-nilly.

If $S'$ is another $G$-space that is conjugate to $S$ via a homeomorphism $\theta\:S'\to S$,
then $\theta$ transports (in the reverse direction) the cocycle $\varphi$ for the action on $S$ to a cocycle $\varphi'$ for the action on $S'$ via
\[
\varphi'(g,x)=\varphi(g,\theta(x))\righttext{for}g\in G,x\in S'.
\]

\begin{ex}
If $\pi\:G\to T$ is any continuous homomorphism, then the map $\varphi$ defined by
\[
\varphi(g,x)=\pi(g)
\]
is a $T$-valued cocycle for the action $G\act S$,
and these are precisely the cocycles that are constant in the second coordinate.
In particular, the map $\varphi(g,x)=1$ is a cocycle, where 1 denotes the identity element of $T$.
\end{ex}

\begin{ex}
If $G=\Z$ then there is a bijection between the set of $T$-valued cocycles for $G\act S$ and the set of continuous maps $\xi\:S\to T$, given by
\[
\xi(x)=\varphi(1,x),
\]
where 1 denotes the identity element of $G$.
We will need to use this, and for convenience we will call $\xi$ the \emph{generating function} of $\varphi$.
\end{ex}

\begin{ex}
If $\varphi\:G\times S\to T$ is a cocycle and $\pi\:T\to R$ is a continuous homomorphism to another locally compact group, then $\pi\circ\varphi$ is an $R$-valued cocycle.
\end{ex}

The cocycle identity is precisely what is needed so that the equation
\begin{equation}\label{induced action}
g\cdot (x,t)=(gx,\varphi(g,x)t)\righttext{for}g\in G,x\in S,t\in T
\end{equation}
defines an action of $G$ on $S\times T$.
Two $T$-valued cocycles $\varphi$ and $\varphi'$ for the action $G\act S$ are \emph{cohomologous} if there is a continuous map $\psi\:S\to T$ such that
\begin{equation}\label{cohomologous}
\varphi'(g,x)=\psi(gx)\varphi(g,x)\psi(x)\inv\righttext{for all}g\in G,x\in S.
\end{equation}
Conversely, starting with a $T$-valued cocycle $\varphi$ for the action $G\act S$
and a continuous map $\psi\:S\to T$,
the map $\varphi'$ defined by \eqref{cohomologous} is also a cocycle for the action $G\act S$
(which is then cohomologous to $\varphi$ by construction).
Moreover,
the respective actions 
$\cdot$ and $\cdot'$ 
of $G$ on $S\times T$ are \emph{conjugate}:
the homeomorphism $\theta$ on $S\times T$ defined by
\[
\theta(x,t)=\bigl(x,\psi(x)t\bigr)
\]
satisfies
\[
g\cdot'\theta(x,t)=\theta\bigl(g\cdot(x,t)\bigr)\righttext{for all}g\in G,x\in S,t\in T.
\]
A cocycle is a \emph{coboundary} if it is cohomologous to the trivial cocycle $\varphi(g,x)=1$.
If the group $T$ is abelian, then the set of $T$-valued cocycles for the action $G\act S$ is an abelian group, the coboundaries form a subgroup, and the set of cohomology classes of cocycles is the quotient group.

\begin{ex}
Let $\varphi,\varphi'$ be $T$-valued cocycles for an action $\Z\act S$,
with respective generating functions $\xi,\xi'$.
Let the $\Z$-action be generated by the homeomorphism $\tau$ on $S$.
Then $\varphi$ and $\varphi'$ are cohomologous if and only if there is a continuous map $\psi\:S\to T$ such that
\[
\xi'(x)=\psi(\tau(x))\xi(x)\psi(x)\inv\righttext{for all}x\in S,
\]
in which case $\psi$ also satisfies \eqref{cohomologous}.
\end{ex}

The following elementary result is presumably folklore, but we could not find it in the literature, so we include the short proof:

\begin{lem}\label{g cbdy}
Let $G\act S$, and let $\varphi\:G\times S\to G$ be a cocycle.
Then the following are equivalent:
\begin{enumerate}
\item
The cocycle $(g,e)\mapsto g$ is a coboundary.

\item
There is a
continuous map $\psi\:S\to G$ such that
\begin{align*}
\psi(gx)&=g\psi(x)\righttext{for all}g\in G,x\in S.
\end{align*}

\item
$S$ is $G$-equivariantly homeomorphic to a space of the form $G\times R$, where $G$ acts by left translation in the first factor.
\end{enumerate}
\end{lem}

\begin{proof}
(1)$\iff$(2) follows immediately from the definitions,
and (3)$\then$(2) is trivial.
Assuming (2),
put $R=\psi\inv(\{1\})$.
It is an elementary exercise to show that
the map $\theta\:S\to G\times R$ defined by
\[
\theta(x)=\bigl(\psi(x),\psi(x)\inv x\bigr)
\]
is a $G$-equivariant homeomorphism,
giving (3).
\end{proof}

We have not seen the following terminology in the literature, but it surely expresses a standard relationship. Since we will need it, we record it formally.

\begin{defn}\label{cohomology conjugate}
Suppose that we have two actions of $G$ on respective spac\-es $S$ and $S'$, with respective $T$-valued cocycles $\varphi$ and $\varphi'$.
We say that the systems $(G,S,\varphi)$ and $(G,S',\varphi')$ are
\emph{cohomology conjugate}
if there is a homeomorphism $\theta\:S'\to S$
that intertwines the actions
and transports $\varphi$ to a cocycle that is cohomologous to $\varphi'$.
\end{defn}

Suppose that $G$ acts on a finite set $S$.
Let the group $T$ be abelian, and write it additively.
Let $\varphi\:G\times S\to T$ be a cocycle.
Then
the function
\[
g\mapsto \sum_{x\in S}\varphi(g,x)
\]
is a cohomology invariant,
because
for any map $\psi\:S\to T$ we have
\begin{align*}
\sum_{x\in S}\bigl(\varphi(g,x)+\psi(gx)-\psi(x)\bigr)
&=\sum_{x\in S}\varphi(g,x)+\sum_{x\in S}\psi(gx)-\sum_{x\in S}\psi(x)
\\&=\sum_{x\in S}\varphi(g,x),
\end{align*}
since $x\mapsto gx$ is a permutation of $S$.
In particular, if $G=\Z$ and $\varphi$ has generating function $\xi\:S\to T$, then the number
\[
\sum_{x\in S}\xi(x)
\]
is a cohomology invariant. We call this number the \emph{signature} of the cocycle $\varphi$.
We will find it useful to record the following consequence, which is surely folklore:

\begin{lem}\label{signature}
Let $\Z\act S$ and $\Z\act S'$,
and let $\varphi$ and $\varphi'$ be $T$-valued cocycles for the respective actions.
If $S$ and $S'$ are finite, $T$ is abelian,
and the cocycles $\varphi$ and $\varphi'$ have different signatures,
then the systems $(\Z,S,\varphi)$ and $(\Z,S',\varphi')$
are not cohomology conjugate in the sense of \defnref{cohomology conjugate}.
\end{lem}

\lemref{Zimmer} below is \cite[4.2.13]{Zimmer}.
Zimmer proved the result in greater generality, involving Borel actions and cocycles, but we restrict ourselves to the discrete case.
We briefly summarize the proof for convenient reference.

\begin{lem}[Zimmer]\label{Zimmer}
Let $G$ and $T$ be discrete groups,
let $H$ be a subgroup of $G$,
and let $\varphi\:G\times G/H\to T$ be a $T$-valued cocycle
for the canonical action by left translation.
Define $\pi_\varphi\:H\to T$ by
\[
\pi_\varphi(h)=\varphi(h,H).
\]
Then $\pi_\varphi$ is a homomorphism,
and moreover the map $\varphi\mapsto \pi_\varphi$
gives a bijection from the set of
cohomology classes
of $T$-valued cocycles for the action $G\curvearrowright G/H$
to the set of
conjugacy classes
of homomorphisms from $H$ to $T$.
\end{lem}

\begin{proof}
It follows immediately from the cocycle identity 
that $\pi_\varphi$ defines a homomorphism.

Now let $\pi\:H\to T$ be a homomorphism.
Choose a cross-section $\eta\:G/H\to G$ such that $\eta(H)=1$,
and define $\varphi_0\:G\times G/H\to H$ by
\[
\varphi_0(g,x)=\eta(gx)\inv g\eta(x).
\]
It is an easy exercise in the definitions to check that $\varphi_0$
is a cocycle with values in $H$,
and that
\[
\pi_{\varphi_0}=\id_H.
\]
Then $\varphi:=\pi\circ\varphi_0\:G\times G/H\to T$
is  a cocycle,
and
one readily checks that
with $\pi_\varphi=\pi$.
Thus the map $\varphi\mapsto \pi_\varphi$ is onto the set of homomorphisms from $H$ to $T$.

Let $\varphi,\varphi'\:G\times G/H\to T$ be cocycles,
with associated homomorphisms $\pi,\pi'$.
Suppose that $\varphi'$ is cohomologous to $\varphi$,
and choose a map $\psi\:G/H\to T$ such that
\[
\varphi'(g,x)=\psi(gx)\varphi(g,x)\psi(x)\inv\righttext{for all}g\in G,x\in G/H.
\]
Then for all $h\in H$ we have
\begin{align*}
\pi'(h)
=\psi(hH)\pi(h)\psi(H)\inv=\ad \psi(H)\circ \pi(h),
\end{align*}
so the element $\psi(H)\in T$ conjugates $\pi$ to $\pi'$.

Conversely, let $t\in T$,
and suppose that $\pi'=\ad t\circ\pi$.
Note that
\begin{align*}
\pi\bigl(\eta(gx)\inv g\eta(x)\bigr)
&=\varphi\bigl(\eta(gx)\inv g\eta(x),H\bigr)
\\&=\varphi\bigl(\eta(gx)\inv,g\eta(x)H\bigr)
\varphi\bigl(g\eta(x),H\bigr)
\\&=\varphi\bigl(\eta(gx),\eta(gx)\inv g\eta(x)H\bigr)\inv
\varphi\bigl(g,\eta(x)H\bigr)
\varphi\bigl(\eta(x),H\bigr)
\\&=\varphi\bigl(\eta(gx),H\bigr)\inv\varphi(g,x)\varphi\bigl(\eta(x),H\bigr),
\end{align*}
because $\eta(gx)\inv g\eta(x)\in H$ and $\eta(x)H=x$,
and similarly for $\pi'$ and $\varphi'$.
Thus
\begin{align*}
&\varphi'(g,x)
\\&\quad=\varphi'\bigl(\eta(gx),H\bigr)t\inv \varphi\bigl(\eta(gx),H\bigr)\inv
\varphi(g,x)\varphi\bigl(\eta(x),H\bigr)t\varphi'\bigl(\eta(x),H\bigr)\inv
\\&\quad=\psi(gx)\varphi(g,x)\psi(x)\inv,
\end{align*}
where $\psi\:G/H\to T$ is defined by
\[
\psi(x)=\varphi'\bigl(\eta(x),H\bigr)t\inv \varphi\bigl(\eta(x),H\bigr)\inv,
\]
and hence $\varphi'$ is cohomologous to $\varphi$.
\end{proof}

\begin{rem}
Note that, in the notation of the above proof,
if we are given a cocycle $\varphi$,
we can explicitly compute how the cocycle $\pi_\varphi\circ\varphi_0$ is cohomologous to $\varphi$:
\begin{align*}
\pi_\varphi\circ\varphi_0(g,x)
&=\varphi\bigl(\eta(gx)\inv g\eta(x),H\bigr)
\\&=\varphi\bigl(\eta(gx),H\bigr)\inv\varphi(g,x)\varphi\bigl(\eta(x),H\bigr)
\\&=\tau(gx)\varphi(g,x)\tau(x)\inv,
\end{align*}
where $\tau\:G/H\to T$ is defined by
\[
\tau(x)=\varphi(\eta(x),H)\inv.
\]
\end{rem}

\begin{rem}\label{Z hom}
With the hypotheses of \lemref{Zimmer}, if $T$ is abelian then the group of cohomology classes of cocycles of $T$-valued cocycles for $G\act G/H$ is isomorphic to the group of homomorphisms from $H$ to $T$.
\end{rem}

The following result is surely standard, but since
we could not find it in the literature and
we need to refer to it later, we give the elementary proof.

\begin{lem}\label{H1Z}
Let $a$ be a positive integer, and let
\[
\Z_a=\Z/a\Z=\{0,1,\dots,a-1\}
\]
be the quotient group.
Let $\Z$ act on $\Z_a$ in the canonical manner, by translation modulo $a$.
For any $c\in\Z$ define $\xi_c\:\Z_a\to\Z$ by
\[
\xi_c(x)=\begin{cases}0\case x<a-1\\c\case x=a-1,\end{cases}
\]
and let $\varphi_c$ be the cocycle with generating function $\xi_c$.
Then $\{\varphi_c:c\in\Z\}$ is
a complete set of representatives for the set of cohomology classes of $\Z$-valued cocycles for the canonical action $\Z\act \Z_a$.
\end{lem}

\begin{proof}
The action of $\Z$ on $\Z_a$ is generated by the permutation $\tau$ of $\Z_a$ given by
\[
\tau(x)=x+1.
\]
In the notation of the proof of \lemref{Zimmer},
choose the cross section $\eta\:\Z_a\to\Z$ to be given by
\[
\eta(k+a\Z)=k\righttext{for}k=0,1,\dots,a-1.
\]
The special $a\Z$-valued cocycle $\varphi_0$ as in the proof of \lemref{Zimmer} has generating function
\[
\varphi_0(1,x)=1+\eta(x)-\eta(\tau(x))
=\begin{cases}0\case x<a-1\\a\case x=a-1.\end{cases}
\]
As in the proof of \lemref{Zimmer} (see also \remref{Z hom}), every $\Z$-valued cocycle is cohomologous to a unique cocycle of the form $\pi\circ\varphi_0$ for a homomorphism $\pi\:a\Z\to\Z$.
The homomorphism $\pi$ is uniquely determined by the number $c=\pi(a)\in\Z$,
and a routine calculation shows that the generating function of the cocycle $\pi\circ\varphi_0$ is given by $\xi_c$.
\end{proof}

With the notation of \lemref{H1Z}, we of course see immediately that for distinct $c$ the cocycles $\varphi_c$ are noncohomologous, since $\varphi_c$ has signature $c$.
But in fact the following corollary (which again is surely folklore) shows that much more is true:

\begin{cor}\label{transitive}
For systems $(\Z,S,\varphi)$, where $\Z\act S$ transitively, $S$ is finite, and $\varphi\:\Z\times S\to\Z$ is a cocycle,
the signature of $\varphi$ is a complete invariant for cohomology conjugacy.
\end{cor}

\begin{proof}
Let $(\Z,S,\varphi)$ be such a system,
and let $S$ have cardinality $a$.
By transitivity this system is cohomology conjugate to a system of the form $(\Z,\Z_a,\varphi')$, where $\Z$ acts on $\Z_a$ by the usual translation modulo $a$.
It follows from \lemref{H1Z} that the cocycle $\varphi'$ is determined up to cohomology by its signature,
and moreover every integer can occur as the signature of some cocycle for this action $\Z\act\Z_a$.
\end{proof}

\section{Cocycles for graph actions}
\label{withphi}

Let $E=(E^0,E^1,r,s)$ be a topological graph in the sense of Katsura \cite{Ka1}, that is, $E^0$ and $E^1$ are locally compact Hausdorff spaces, $r\:E^1\to E^0$ is continuous, and $s\:E^1\to E^0$ is a local homeomorphism.

In \cite[Section~2]{Ka1}, Katsura constructs a correspondence $X=X_E$ over the commutative $C^*$-algebra $A:=C_0(E^0)$ as the completion of the pre-correspondence $C_c(E^1)$, with operations
defined for $a\in A$ and $x,y\in C_c(E^1)$ by
\begin{align*}
(a\cdot x)(e)&=a(r(e))x(e)
\\
(x\cdot a)(e)&=x(e)a(s(e))
\\
\<x,y\>_A(v)&=\sum_{s(e)=v}\bar{x(e)}y(e).
\end{align*}
We will call $X$ the \emph{graph correspondence} of $E$.

Katsura defines $C^*(E)$ as the Cuntz-Pimsner algebra $\OO_X$ \cite[Definition~2.10]{Ka1}.
Here we use the conventions of \cite[Definition~3.5]{Kat04} for Cuntz-Pimsner algebras.

Let $G$ be a locally compact group acting continuously on $E$ in the sense of \cite[Section~3]{dkq}, that is, $G$ acts in the usual way by homeomorphisms on the spaces $E^0$ and $E^1$, and for each $g\in G$ the maps $e\mapsto ge$ on edges and $v\mapsto gv$ on vertices constitute an automorphism of the topological graph $E$.

Let $\alpha$ denote the associated action of $G$ on $A$:
\[
\alpha_g(a)(v)=a(g\inv v)\righttext{for}g\in G,a\in A,v\in E^0.
\]
By \cite[Proposition~5.4]{dkq}, we can define an $\alpha$-compatible action $\gamma$ of $G$ on the graph correspondence $X$ via
\[
\gamma_g(x)(e)=x(g\inv e)\righttext{for}g\in G,x\in C_c(E^1),e\in E^1.
\]
Here we use the conventions of \cite[Definition~3.1]{taco} for actions on correspondences.
By \cite[Proposition~3.5]{taco},
there is a $C^*$-correspondence $Y:=X\rtimes_\gamma G$ over the crossed product $B:=A\rtimes_\alpha G$,
which contains $C_c(G,X)$ as a dense subspace,
and which satisfies
\begin{align*}
b\cdot \xi(g)&=\int_G b(h)\cdot \gamma_h\bigl(\xi(h\inv g)\bigr)\,dh
\\
\xi\cdot b(g)&=\int_G \xi(h)\cdot \alpha_h\bigl(b(h\inv g)\bigr)\,dh
\\
\<\xi,\eta\>_{A\rtimes_\alpha G}(g)&=\int_G \alpha_{h\inv}\bigl(\<\xi(h),\eta(hg)\>_A\bigr)\,dh
\end{align*}
for $b\in B$,  $\xi,\eta\in C_c(G,X)$, and $g\in G$.
We call $Y$ the \emph{crossed product} of the action $(X,G)$.
This correspondence is both \emph{full} in the sense that $\clspn \<Y,Y\>_B=B$, and \emph{nondegenerate} in the sense that $BY=Y$.
The left $B$-module multiplication is given by a homomorphism $\phi_Y\:B\to \LL(Y)=M(\KK(Y))$,
which is the integrated form of a covariant pair $(\pi,U)$,
where $\pi\:A\to \LL(Y)$ is the nondegenerate representation
determined by
\begin{equation}\label{left A}
\bigl(\pi(a)\xi\bigr)(h)=a\cdot \xi(h)\righttext{for}a\in A,\xi\in C_c(G,X),h\in G,
\end{equation}
and $U\:G\to \LL(Y)$ is the strongly continuous unitary representation determined by
\begin{equation}\label{U}
(U_g\xi)(h)=\gamma_g(\xi(g\inv h))\righttext{for}g\in G,\xi\in C_c(G,X),h\in G
\end{equation}
(see the proof of \cite[Proposition~3.5]{taco}).

We will want to compute with the $B$-correspondence $Y$ using two-variable functions.
Since $A=C_0(E^0)$,
we can identify the crossed product $A\rtimes_\alpha G$ as a completion of the convolution $*$-algebra $C_c(E^0\times G)$ with operations
\begin{align}
(b*c)(v,g)&=\int_G b(v,h)c(h\inv v,h\inv g)\,dh
\label{convolution two}
\\
b^*(v,g)&=\Delta(g\inv)\bar{b(g\inv v,g\inv)}
\label{star two}
\end{align}
for $b,c\in C_c(E^0\times G)$ and $(v,g)\in E^0\times G$.
(See, for example, \cite[page~53]{danacrossed}.)
Since we will play a similar game with $X\rtimes_\gamma G$,
we pause to provide a little detail on how \eqref{convolution two}--\eqref{star two} are derived from the usual operations on the convolution algebra $C_c(G,A)$, given by
\begin{align}
(b*c)(g)&=\int_G b(h)\alpha_h(c(h\inv g))\,dh
\label{convolution one}
\\
b^*(g)&=\Delta(g\inv)\alpha_g(b(g\inv))^*.
\label{star one}
\end{align}
We do it for \eqref{convolution one}; it is much easier for \eqref{star one}.
Using the embeddings
\[
C_c(E^0\times G)\subset C_c(G,C_c(E^0))\subset C_c(G,A),
\]
for $b,c\in C_c(E^0\times G)$ and $(v,g)\in E^0\times G$, we have
\begin{align*}
(b*c)(v,g)
&=(b*c)(g)(v)
\\&=\int_G b(h)\alpha_h(c(h\inv g))\,dh(v)
\\&\overset{(*)}=\int_G \bigl(b(h)\alpha_h(c(h\inv g))\bigr)(v)\,dh
\\&=\int_G b(h)(v)\alpha_h(c(h\inv g))(v)\,dh
\\&=\int_G b(v,h)c(h\inv g)(h\inv v)\,dh
\\&=\int_G b(v,h)c(h\inv v,h\inv g)\,dh.
\end{align*}
The point is that at the equality $(*)$ we are using that in the line above we have a norm-convergent integral of a continuous $A$-valued function with compact support, and evaluation at $v$ is a bounded linear functional.

Now we argue similarly for
\[
C_c(E^1\times G)\subset C_c(G,C_c(E^1))\subset C_c(G,X),
\]
where in a couple of computations we will have a norm-convergent integral of an $X$-valued function with compact support, and we use the property that evaluation at an edge $e$ is a bounded linear functional on $X$, since on $C_c(E^1)$ the uniform norm is less than the norm from the Hilbert $A$-module $X$.
For $b\in C_c(E^0\times G)$, $\xi,\eta\in C_c(E^1\times G)$, $(e,g)\in E^1\times G$, and $v\in E^0$, we have
\begin{equation}
\begin{split}
(b\cdot \xi)(e,g)&=\int_G b(r(e),h)\xi(h\inv e,h\inv g)\,dh,\\
(\xi\cdot b)(e,g)&=\int_G \xi(e,h)b(h\inv s(e),h\inv g)\,dh,\ \text{ and}\\
\<\xi,\eta\>_{A\rtimes_\alpha G}(v,g)&=\int_G \sum_{s(e)=hv} \bar{\xi(e,h)}\eta(e,hg)\,dh.
\end{split}
\end{equation}
Indeed,
\begin{equation*}
\begin{split}
(b\cdot \xi)(e,g)
&=(b\cdot \xi)(g)(e)
\\&=\int_G b(h)\cdot \gamma_h(\xi(h\inv g))\,dh(e)
\\&=\int_G \bigl(b(h)\cdot \gamma_h(\xi(h\inv g))\bigr)(e)\,dh
\\&=\int_G b(h)(r(e))\gamma_h(\xi(h\inv g))(e)\,dh
\\&=\int_G b(r(e),h)\xi(h\inv g)(h\inv e)\,dh
\\&=\int_G b(r(e),h)\xi(h\inv e,h\inv g)\,dh,
\end{split}
\end{equation*}
\begin{equation*}
\begin{split}
(\xi\cdot b)(e,g)
&=(\xi\cdot b)(g)(e)
\\&=\int_G \xi(h)\cdot \alpha_h(b(h\inv g))\,dh(e)
\\&=\int_G \bigl(\xi(h)\cdot \alpha_h(b(h\inv g))\bigr)(e)\,dh
\\&=\int_G \xi(h)(e)\alpha_h(b(h\inv g))(s(e))\,dh
\\&=\int_G \xi(e,h)b(h\inv g)(h\inv s(e))\,dh
\\&=\int_G \xi(e,h)b(h\inv s(e),h\inv g)\,dh,
\end{split}
\end{equation*}
and
\begin{equation*}
\begin{split}
\<\xi,\eta\>_{A\rtimes_\alpha G}(v,g)
&=\<\xi,\eta\>_{A\rtimes_\alpha G}(g)(v)
\\&=\int_G \alpha_{h\inv}\bigl(\<\xi(h),\eta(hg)\>_A\bigr)\,dh(v)
\\&=\int_G \bigl(\alpha_{h\inv}\bigl(\<\xi(h),\eta(hg)\>_A\bigr)(v)\bigr)\,dh
\\&=\int_G \<\xi(h),\eta(hg)\>_A(hv)\,dh
\\&\overset{(*)}=\int_G \sum_{s(e)=hv} \bar{\xi(h)(e)}\eta(hg)(e)\,dh
\\&=\int_G \sum_{s(e)=hv} \bar{\xi(e,h)}\eta(e,hg)\,dh,
\end{split}
\end{equation*}
where in the equality at $(*)$ the sum is finite by \cite[Lemma~1.4]{Ka1}, since $\xi(h)\in C_c(E^1)$.

We can also compute with the covariant pair $(\pi,U)$ of \eqref{left A} and \eqref{U} using two-variable functions:
for $a\in A=C_0(E^0)$, $\xi\in C_c(E^1\times G)$, $g,h\in G$, and $e\in E^1$ we have
\begin{align}
\bigl(\pi(a)\xi\bigr)(e,h)&=a(r(e))\xi(e,h)
\label{left A two}
\\
(U_g\xi)(e,h)&=\xi(g\inv e,g\inv h).
\label{U two}
\end{align}

\begin{lem}\label{ind lim}
The inductive limit topology on $C_c(E^1\times G)$ is stronger than the norm topology from $Y$.
\end{lem}

\begin{proof}
It suffices to show that if $\{\xi_i\}$ is a net in $C_c(E^1\times G)$
converging uniformly to 0 and
such the supports of the $\xi_i$'s are all contained in some fixed compact set $K$,
then
\[
\|\<\xi_i,\xi_i\>_B\|\to 0.
\]
Choose compact sets $K_1\subset E^1$ and $K_2\subset G$ such that
$K\subset K_1\times K_2$.
By the elementary \lemref{finite} below, we can choose $n\in\N$ such that
for all $v\in E^0$
the set $s\inv(v)\cap K_1$ has at most $n$ elements.
Let $\epsilon>0$,
and choose $i_0$ such that for all $i\ge i_0$ we have
\[
|\xi_i(e,g)|<\epsilon\righttext{for all}(e,g)\in E^1\times G.
\]
For all $i\ge i_0$ and $(v,g)\in E^0\times G$,
\begin{align*}
\<\xi_i,\xi_i\>_B(v,g)
&=\int_G \sum_{s(e)=hv} \bar{\xi_i(e,h)}\xi_i(e,hg)\,dh
\\&=\int_{K_2} \sum_{s(e)=hv} \bar{\xi_i(e,h)}\xi_i(e,hg)\,dh,
\end{align*}
which has absolute value bounded above by $n\epsilon^2$ times the measure of $K_2$,
and this suffices to show that $\|\<\xi_i,\xi_i\>_B\|\to 0$ uniformly.
\end{proof}

In the above proof we used the following elementary lemma, which we could not find in the literature (although it is similar to \cite[Corollary~3.9 (2)]{dkq}):

\begin{lem}\label{finite}
For any compact set $K\subset E^1$ there is a positive integer $n$ such that for all $v\in E^0$ the set $s\inv(v)\cap K$ has at most $n$ elements.
\end{lem}

\begin{proof}
Let $L=s(K)$, a compact subset of $E^0$.
By \cite[Lemma~1.4]{Ka1}, each $v\in E^0$
has a neighborhood $U_v$ such that
for some positive integer $n_v$
the set $K\cap s\inv(U_v)$ has at most $n_v$ elements.
Covering $L$ by finitely many $U_v$'s gives the lemma.
\end{proof}

We will soon modify the $B$-correspondence $Y$ using a cocycle.
The following definition of cocycle generalizes that of \cite[Section~2]{EP}, where the authors consider discrete groups acting on finite graphs.

\begin{defn}\label{cocycle}
A \emph{cocycle} for the action of $G$ on the topological graph $E$ is a 
cocycle $\varphi$ for the action of $G$ on the edge space $E^1$
that also satisfies the \emph{vertex condition}
\begin{equation}\label{vertex condition}
\varphi(g,e)s(e)=gs(e)\righttext{for all}g\in G,e\in E^1.
\end{equation}
\end{defn}

In \cite[(2.3.1)]{EP} (for finite graphs), Exel and Pardo impose a stronger version of \eqref{vertex condition}, namely $\varphi(g,e)v=gv$ for all $g\in G$, $e\in E^1$, and $v\in E^0$; our weakened version above is all that is needed, and allows for greater flexibility.
For example, the elementary theory of cohomology for cocycles
(see \secref{cohom sec})
would be significantly hampered with the Exel-Pardo version.

\begin{rem}
Note that \eqref{vertex condition} implies that for all $(g,e)\in G\times E^1$ the product $g\inv \varphi(g,e)$ lies in the isotropy subgroup $G_{s(e)}$ of $G$ at the vertex $s(e)$. Thus, the existence of nontrivial cocycles, i.e., other than the map $(g,e)\mapsto g$, depends upon having nontrivial isotropy of the action on vertices.
\end{rem}

Let $\varphi$ be a cocycle for  the action of $G$ on $E$. We use $\varphi$ to modify the $B$-correspondence $Y$ as follows:
we keep the same structure as a Hilbert $B$-module,
as well as the same left $A$-module action determined by \eqref{left A two}.
In the following we 
use the technique of \eqref{induced action}
to define
an action of $G$ on the Hilbert $B$-module $Y$:
for $g\in G$ and $\xi\in C_c(E^1\times G)$ define the function $V_g\xi\in C_c(E^1\times G)$ by
\begin{equation}\label{V}
(V_g\xi)(e,h)=\xi(g\inv e,\varphi(g\inv,e)h).
\end{equation}

\begin{prop}\label{G varphi}
\begin{enumerate}
\item
The map $g\mapsto V_g$ given by \eqref{V} is a strong\-ly continuous unitary representation of $G$ on the Hilbert $B$-module $Y$.

\item
With $\pi$ as in \eqref{left A two} and $V$ as above, the pair $(\pi,V)$ is a covariant representation of the system $(A,G,\alpha)$ on the Hilbert $B$-module $Y$.
\end{enumerate}
\end{prop}

\begin{proof}
(1)
First note that the map $g\mapsto V_g$ is multiplicative from $G$ into the set of linear operators on $C_c(E^1\times G)$:
\begin{align*}
(V_gV_h\xi)(e,k)
&=(V_h\xi)\bigl(g\inv e,\varphi(g\inv,e)k\bigr)
\\&=\xi\bigl(h\inv g\inv e,\varphi(h\inv,g\inv e)\varphi(g\inv,e)k\bigr)
\\&=\xi\bigl(h\inv g\inv e,\varphi(h\inv g\inv,e)k\bigr)
\\&=(V_{gh}\xi)(e,k).
\end{align*}
Since $V_{1_G}\xi=\xi$ for all $\xi\in C_c(E^1\times G)$
(where $1_G$ here denotes the identity element of $G$),
we deduce that $V$ is a homomorphism from $G$ to the group of invertible linear operators on $C_c(E^1\times G)$.

Now we show that 
 the inner products on $C_c(E^1\times G)$ are preserved by each $V_g$:
\begin{align*}
&\<V_g\xi,V_g\eta\>_B(v,h)
\\&\quad=\int_G \sum_{s(e)=kv}\bar{(V_g\xi)(e,k)}(V_g\eta)(e,kh)\,dk
\\&\quad=\int_G \sum_{s(e)=kv}
\bar{\xi(g\inv e,\varphi(g\inv,e)k)}\eta(g\inv e,\varphi(g\inv,e)kh)\,dk
\\&\quad=\int_G \sum_{s(e)=kv}
\bar{\xi(e,k)}\eta(e,kh)\,dk
\intertext{\big[after $e\mapsto ge$ and $k\mapsto \varphi(g\inv,e)\inv k$, since a short computation using the cocycle identity and \eqref{vertex condition} shows that $s(e)=kv$ if and only if $s(g\inv e)=\varphi(g\inv,e)\inv kv$\big]
}
&\quad=\<\xi,\eta\>_B(v,h).
\end{align*}
In particular, $V_g$ is isometric on $C_c(E^1\times G)$, and hence extends uniquely to an isometry on the completion $Y$; moreover, since $V_g$  maps $C_c(E^1\times G)$ onto itself, this extension, which we continue to denote by $V_g$, is in fact an isometric linear map of $Y$ onto itself.
Upon taking limits we see that these extensions still satisfy $V_gV_h=V_{gh}$ for all $g,h\in G$.

For $b\in C_c(E^0\times G)$ we have
\begin{align*}
&\bigl(V_g(\xi\cdot b)\bigr)(e,h)
\\&\quad=(\xi\cdot b)(g\inv e,\varphi(g\inv,e)h)
\\&\quad=\int_G \xi(g\inv e,k)b\bigl(k\inv s(g\inv e),k\inv \varphi(g\inv,e)h\bigr)\,dk
\\&\quad=\int_G \xi(g\inv e,\varphi(g\inv,e)k)b\bigl(k\inv \varphi(g\inv,e)\inv s(g\inv e),k\inv h\bigr)\,dk
\\&\hspace{1in}\text{(after $k\mapsto \varphi(g\inv ,e)k$)}
\\&\quad=\int_G \xi(g\inv e,\varphi(g\inv,e)k)b\bigl(k\inv \varphi(g,g\inv e)s(g\inv e),k\inv h\bigr)\,dk
\\&\quad=\int_G \xi(g\inv e,\varphi(g\inv,e)k)b\bigl(k\inv gs(g\inv e),k\inv h\bigr)\,dk
\righttext{(by \eqref{vertex condition})}
\\&\quad=\int_G (V_g\xi)(e,k)b\bigl(k\inv s(e),k\inv h\bigr)\,dk
\\&\quad=\bigl((V_g\xi)\cdot b\bigr)(e,h).
\end{align*}
Thus by continuity the map $V_g$ on $Y$ is right $B$-linear,
and this combined with its other properties makes it a unitary operator on the Hilbert $B$-module $Y$ \cite[Theorem~3.5]{lance}.

For the strong continuity,
by uniform boundedness
it suffices to show that if $\xi\in C_c(E^1\times G)$ and $g_i\to 1$ in $G$ then $\|V_{g_i}\xi-\xi\|\to 0$.
Arguing by contradiction, we can replace $\{g_i\}$ by a subnet and relabel so that no subnet of $\{\|V_{g_i}\xi-\xi\|\}$ converges to 0.
Again replacing $\{g_i\}$ by a subnet, we can suppose that the $g_i$'s are all contained in some compact neighborhood $U$ of 1.
It then follows from continuity of the operations, and of the function $\varphi$, that the supports of the functions $V_{v_i}\xi$'s are all contained in some fixed compact set $K\subset E^1\times G$.
Then by \lemref{ind lim} it suffices to show that
$V_{g_i}\xi\to \xi$ uniformly.
Arguing by contradiction, we can replace by a subnet so that no subnet of $V_{g_i}\xi$ converges uniformly to $\xi$.
Then we can find $\epsilon>0$ such that,
after again replacing by a subnet,
for all $i$ there exists $(e_i,h_i)\in E^1\times G$ such that
\[
|V_{g_i}\xi(e_i,h_i)-\xi(e_i,h_i)|\ge \epsilon.
\]
In particular, we must have $(e_i,h_i)\in K$ for all $i$,
so that after replacing by a subnet again we have $(e_i,h_i)\to (e,h)$ for some $(e,h)\in E^1\times G$.
But then by continuity we have
\[
\epsilon\le |V_{g_i}\xi(e_i,h_i)-\xi(e_i,h_i)|\to |\xi(e,h)-\xi(e,h)|=0,
\]
which is a contradiction.

(2)
It suffices to show that for all $g\in G$, $a\in A$, and $\xi\in C_c(E^1\times G)$ we have
$V_g\pi(a)\xi=\pi(\alpha_g(a))V_g\xi$,
and we check this by evaluating at an arbitrary pair $(e,h)\in E^1\times G$:
\begin{align*}
(V_g\pi(a)\xi)(e,h)
&=(\pi(a)\xi)\bigl(g\inv e,\varphi(g\inv,e)h\bigr)
\\&=a\bigl(r(g\inv e)\bigr)\xi\bigl(g\inv e,\varphi(g\inv,e)h\bigr)
\\&=a\bigl(g\inv r(e)\bigr)(V_g\xi)(e,h)
\\&=\alpha_g(a)(r(e))(V_g\xi)(e,h)
\\&=\bigl(\pi(\alpha_g(a))V_g\xi\bigr)(e,h).
\qedhere
\end{align*}
\end{proof}

\begin{defn}
The integrated form of the covariant representation $(\pi,V)$ of \propref{G varphi} is a nondegenerate representation  of $B$ in $\LL(Y)$, giving $Y$ the structure of a $B$-correspondence
that
we
denote by $\yphi$. 
\end{defn}

\begin{ex}
For the special cocycle $(g,e)\mapsto g$ the correspondence $\yphi$ reduces to the crossed product $Y=X\rtimes_\gamma G$.
\end{ex}

It will be useful to handle the left module action of $B$ on $\yphi$ in terms of two-variable functions:
for $b\in C_c(E^0\times G)\subset B$ and $\xi\in C_c(E^1\times G)\subset \yphi$,
the function $b\cdot \xi\in C_c(E^1\times G)$ is given by
\begin{align*}
(b\cdot \xi)(e,g)
&=\bigl(\int_G \bigl(i_A(b(h))i_G(h)\cdot \xi\bigr)\,dh\bigr)(e,g)
\\&=\int_G \bigl(i_A(b(h))i_G(h)\cdot \xi\bigr)(e,g)\,dh
\righttext{by \lemref{ind lim}}
\\&=\int_G b(h)(r(e))\bigl(i_G(h)\cdot \xi\bigr)(e,g)\,dh
\righttext{by \eqref{left A two}}
\\&=\int_G b(r(e),h)(V_h\xi)(e,g)\,dh
\\&=\int_G b(r(e),h)\xi(h\inv e,\varphi(h\inv,e)g)\,dh.
\end{align*}

Following \cite{Kat04} we can associate two $C^*$-algebras to the correspondence $\yphi$:  the Toeplitz algebra  $\ty$ and the Cuntz-Pimsner algebra  $\oy$. It will follow from Corollary \ref{universal O} (or Corollary \ref{alt EP alg}) that $\oy$ is isomorphic to the $C^*$-algebra $\OO_{G, E}$ associated to $(E, G, \varphi)$ by Exel and Pardo in \cite{EP} in the case where $E$ is finite and sourceless, $G$ is discrete and $\varphi$ satisfies $\varphi(g, e) v = v$ for all $(g, e)\in G\times E^1$ and $v\in E^0$.

\begin{defn}
We will call $\oy$ the {\it Exel-Pardo algebra} associated to the system $(E, G, \varphi)$.
\end{defn}

\begin{rem}
Exel and Pardo show that in the case they consider, $\OO_{G, E}$ is nuclear whenever $G$ is amenable (cf.\ \cite[Corollary 10.12]{EP}). In our more general context, assume that the action of the locally compact group $G$ on the locally compact Hausdorff space $E^0$ is amenable in the sense of Anantharaman-Delaroche (see \cite[Section 2]{AD02}). Then $B=C_0(E^0)\rtimes_\alpha G$ is nuclear \cite[Theorem 5.3]{AD02}, and it then follows from \cite[Theorem 7.2 and Corollary 7.4]{Kat04} that $\ty$ and $\oy$ are nuclear. If $G$ is discrete, then $B$ is nuclear if and only if the action of $G$ on $E^0$ is amenable, cf.\ \cite[Th{\'e}or{\`e}me 4.5]{AD87}. Hence, when $G$ is discrete,  \cite[Theorem 7.2]{Kat04} gives that $\ty$ is nuclear if and only if $G$ acts amenably on $E^0$.
\end{rem}

\begin{rem}
By \cite[Proposition~3.2]{enchilada},
there is also a $C^*$-cor\-res\-pon\-dence $Y_r:=X\rtimes_{\gamma,r} G$ over the reduced crossed product $B_r:=C_0(E^0)\rtimes_{\alpha, r} G$,
which contains $C_c(G,X)$ as a dense subspace and is constructed in a similar way as $Y$. To be able to talk about the reduced $C^*$-correspondence $Y_r^\varphi$, i.e., to define a left action of $B_r$ on $Y_r$ involving $\varphi$, one will have to find out if $ \pi \times V$ factors through $B_r$ in general.
If $G$ acts amenably on $E^0$, then $B=B_r$ (cf.\ \cite[Theorem 5.3]{AD02}), so the problem does not show up in this case.
\end{rem}

\begin{rem}\label{category}
In \cite[Section~2]{EP}, Exel and Pardo show how to extend the action and cocycle to the set $E^*$ of finite paths, and it is clear that their proof works whenever $E$ is a directed graph and $G$ is discrete. 
We see a way to carry this further,
to form a sort of 
Zappa-Sz{\'e}p product of $E^*$ by $G$ with respect to $\varphi$,
and thereby obtain a new category of paths
$E^*\rtimes^{\varphi} G$
in the sense of Spielberg \cite{Spi11},
except that right cancellativity will not hold in general,
and a little bit of work is necessary to force the category to have no inverses.
Several natural questions arise:
is the algebra $C^*(E^*\rtimes^{\varphi} G)$ that Spielberg's theory associates to this category of paths
isomorphic (or related) to the Toeplitz algebra $\ty$?
And then is a suitable quotient of $C^*(E^*\rtimes^\varphi G)$ isomorphic to the Cuntz-Pimsner algebra $\oy$?
We plan to pursue this in subsequent work.
\end{rem}

\section{Cohomology for graph cocycles}\label{cohom sec}

Throughout this section $G$ will be a locally compact group acting on a topological graph $E$.

\begin{defn}
Let $\varphi,\varphi'$ be cocycles for the action $G\act E$.
We say $\varphi$ and $\varphi'$ are \emph{cohomologous} if
there is a continuous function $\psi\:E^1\to G$ such that for all $g\in E$ and $e\in E^1$ we have
\begin{align}
\varphi'(g,e)&=\psi(ge)\varphi(g,e)\psi(e)\inv\label{cohomologous graph}
\\
\psi(e)s(e)&=s(e).\label{cochain}
\end{align}
\end{defn}

Note that \eqref{cohomologous graph} just says that $\varphi$ and $\varphi'$ are cohomologous as cocycles for the action $G\act E^1$.
The extra condition \eqref{cochain} is necessary to make the theory work for actions on topological graphs.

\begin{lem}\label{is a cocycle}
If $\varphi$ is a cocycle for the action of $G$ on the topological graph $E$
and $\psi\:E^1\to G$ is a continuous map satisfying \eqref{cochain},
then the map $\varphi'\:E^1\times G\to G$ defined by \eqref{cohomologous graph} is also a cocycle for the action of $G$ on $E$.
\end{lem}

\begin{proof}
As we mentioned above, the cocycle identity 
holds for $\varphi'$ by the standard theory of actions on spaces (and is a routine computation).
We verify \eqref{vertex condition}:
\begin{align*}
\varphi'(g,e)s(e)
&=\psi(ge)\varphi(g,e)\psi(e)\inv s(e)
\\&=\psi(ge)\varphi(g,e)s(e)
\\&=\psi(ge)gs(e)
\\&=\psi(ge)s(ge)
\\&=s(ge)
\\&=gs(e).
\qedhere
\end{align*}
\end{proof}

In the general theory of cocycles for actions on spaces,
the constant function $(g,e)\mapsto 1$ is a cocycle (where 1 here denotes the identity element of $G$). But not necessarily for the action of $G$ on the topological graph $E$:

\begin{lem}\label{fix}
For an action of $G$ on a topological graph $E$, the following are equivalent:
\begin{enumerate}
\item The constant function $(g,e)\mapsto 1$ is a cocycle for the action $G\curvearrowright E$.

\item 
$gs(e)=s(e)$ for all $(g,e)\in G\times E^1$.

\item 
In \defnref{cocycle}, the axiom \eqref{vertex condition} is redundant.
\end{enumerate}
\end{lem}

\begin{proof}
(2) trivially implies (3), which in turn trivially implies (1).
Assume (1).
Then for all $g\in E$ and $e\in E^1$ we have
\[
gs(e)=1_Gs(e)=s(e),
\]
giving (2).
\end{proof}

\begin{defn}
We say that an action of $G$ on a topological graph $E$
\emph{fixes sources} if it
satisfies the equivalent conditions (1)--(3) in Lem\-ma~\ref{fix}.
\end{defn}

\begin{cor}\label{trivial}
If the action $G\curvearrowright E$ 
fixes sources,
and
if $\psi\:E^1\to G$ is a continuous map satisfying \eqref{cochain}, then the map $\varphi\:G\times E^1\to G$ defined by
\begin{equation}\label{coboundary}
\varphi(g,e)=\psi(ge)\psi(e)\inv
\end{equation}
is a cocycle for the action $G\curvearrowright E$.
\end{cor}

\begin{defn}
If the action $G\curvearrowright E$ fixes sources,
a cocycle $\varphi$ as in 
\eqref{coboundary}
is a \emph{coboundary} for the action $G\curvearrowright E$.
\end{defn}

\begin{rem}
Thus, when the action $G\curvearrowright E$ fixes sources,
coboundaries are precisely the cocycles that are cohomologous to the cocycle taking the constant value $1$.
\end{rem}

Cohomologous cocycles give isomorphic correspondences:

\begin{thm} \label{coho cocy}
If $\varphi$ and $\varphi'$ are cohomologous cocycles for the action $G\curvearrowright E$, then the $B$-correspondences $\yphi$ and $Y^{\varphi'}$ are isomorphic.
\end{thm}

\begin{proof}
Let $\varphi'(e,g)=\psi(eg)\varphi(e,g)\psi(e)\inv$ for a continuous map $\psi:E^1\to G$ satisfying \eqref{cochain}.
We will construct an isomorphism $\Phi\:Y^{\varphi}\to Y^{\varphi'}$.
To begin, we define $\Phi$ as a linear map on $C_c(E^1\times G)$ by
\[
(\Phi\xi)(e,g)=\xi(e,\psi(e)\inv g)\righttext{for}\xi\in C_c(E^1\times G).
\]
We show that $\Phi$ preserves inner products:
\begin{align*}
\<\Phi\xi,\Phi\eta\>(v,g)
&=\int_G \sum_{s(e)=hv} \bar{(\Phi\xi)(e,h)}(\Phi\eta)(e,hg)\,dh
\\&=\int_G \sum_{s(e)=hv} \bar{\xi(e,\psi(e)\inv h)}(\eta(e,\psi(e)\inv hg)\,dh
\\&=\int_G \sum_{s(e)=hv} \bar{\xi(e,h)}(\eta(e,hg)\,dh,
\intertext{after $h\mapsto \psi(e)h$, since by \eqref{cochain} $s(e)=hv$ if and only if $s(e)=\psi(e)hv$,}
&=\<\xi,\eta\>(v,g).
\end{align*}
Thus $\Phi$ extends uniquely to an isometric linear operator on the Hilbert $B$-module $Y$.
The following computation implies that $\Phi$ is right $B$-linear:
for $\xi\in C_c(E^1\times G)$ and $b\in C_c(E^0\times G)$ we have
\begin{align*}
\bigl(\Phi(\xi\cdot b)\bigr)(e,g)
&=(\xi\cdot b)(e,\psi(e)\inv g)
\\&=\int_G \xi(e,h)b(h\inv s(e),h\inv \psi(e)\inv g)\,dh
\\&=\int_G \xi(e,\psi(e)\inv h)b(h\inv \psi(e)s(e),h\inv g)\,dh,
\\&\hspace{1in}\text{after $h\mapsto \psi(e)\inv h$}
\\&=\int_G \xi(e,\psi(e)\inv h)b(h\inv s(e),h\inv g)\,dh,
\righttext{(by \eqref{cochain})}
\\&=\int_G (\Phi\xi)(e,h)b(h\inv s(e),h\inv g)\,dh
\\&=\bigl((\Phi\xi)\cdot b\bigr)(e,g).
\end{align*}
This combined with the other properties of $\Phi$ makes it a unitary map from the Hilbert $B$-module $\yphi$ to the Hilbert $B$-module $Y^{\varphi'}$
\cite[Theorem~3.5]{lance}.

Then the following computation implies that $\Phi$ is left $B$-linear,
from which the theorem will follow:
\begin{align*}
\bigl(\Phi(b\cdot \xi)\bigr)(e,g)
&=(b\cdot \xi)(e,\psi(e)\inv g)
\\&=\int_G b(r(e),h)\xi(h\inv e,\varphi(h\inv,e) \psi(e)\inv g)\,dh
\\&=\int_G b(r(e),h)\xi(h\inv e,\psi(h\inv e)\inv \varphi'(h\inv,e)g)\,dh
\\&=\int_G b(r(e),h)(\Phi\xi)(h\inv e,\varphi'(h\inv,e)g)\,dh
\\&=\bigl(b\cdot (\Phi\xi)\bigr)(e,g).
\qedhere
\end{align*}
\end{proof}

As an immediate consequence of \thmref{coho cocy}, we get:
\begin{cor}\label{cohomology conjugate 2} 
Assume that  $G$ also acts on another topological graph $F=(F^0, F^1, r', s')$ and that   $\varphi$ and $\varphi'$ are cocycles for $G\act E$ and $G\act F$, respectively. If $(E, G, \varphi)$ and $(F, G, \varphi')$ are
\emph{cohomology conjugate} in the sense that 
there exist $G$-equivariant homeomorphisms $\theta_j\:F^j\to E^j$ for $j=0,1$
such that $r\circ \theta_1= \theta_0\circ r'$, $s\circ \theta_1= \theta_0\circ s'$, and the map $(g,f) \to \varphi(g, \theta_1(f))$ is a cocycle for $G\act F$ that is cohomologous to $\varphi'$,
then $\yphi$ is isomorphic to $Y^{\varphi'}$, and it follows that $\TT_{\yphi}$ \(resp.\  $\OO_{\yphi}$\) is isomorphic to $\TT_{Y^{\varphi'}}$ \(resp.\ $\OO_{Y^{\varphi'}}$\).
\end{cor}

In the following proposition we consider the questions of whether the cocycle $(g,e)\mapsto g$ can be a coboundary for a graph action.

\begin{prop}\label{g cby}
Suppose the action $G\curvearrowright E$ fixes sources.
Then the following are equivalent:
\begin{enumerate}
\item
The cocycle $(g,e)\mapsto g$ is a coboundary.

\item
There is a
continuous map $\psi\:E^1\to G$ such that
for all $g\in G$ and $e\in E^1$
we have
\begin{align*}
\psi(ge)&=g\psi(e).
\end{align*}

\item
$E^1$ is $G$-equivariantly homeomorphic to 
$G\times \Omega$ for some space $\Omega$, where $G$ acts by left translation in the first factor.
\end{enumerate}
\end{prop}

\begin{proof}
Since the action fixes sources, we have
$gs(e)=s(e)$ for all $g\in G$ and $e\in E^1$,
and consequently
it is easy to see that $(g,e)\mapsto g$ is a coboundary for the action of $G$ on the topological graph $E$
if and only if it is a coboundary for the action of $G$ on the space $E^1$,
so the result follows immediately from \propref{g cbdy}.
\end{proof}

\begin{rem}
Suppose that the action
$G\curvearrowright E$ 
fixes sources
and that the map $(g,e)\mapsto g$ is a coboundary.
In view of Proposition \ref{g cby}, we may assume that $E^1=G\times\Omega$ and $g(h,x)=(gh,x)$ for all $g,h\in G$ and $x\in\Omega$.
It is interesting to examine 
the range and source maps of $E$.
Define continuous maps $\sigma,\rho\:\Omega\to E^0$ by
\begin{align*}
\sigma(x)&=s(1,x)
\\
\rho(x)&=r(1,x).
\end{align*}
Then for all $(g,x)\in G\times \Omega$ we have
\[
s(g,x)
=gs(1,x)
=s(1,x)
=\sigma(x).
\]
On the other hand, for the range map we have
\[
r(g,x)=gr(1,x)=g\rho(x).
\]
It is tempting to conjecture that much more can be said about this situation. 
As a kind of converse, let  $\rho, \sigma\: \Omega\to E^0$ be continuous maps between some locally compact Hausdorff spaces $\Omega$ and $E^0$. Assume that $\sigma$ is a local homeomorphism and that  $G$ is a discrete group acting by homeomorphisms on $E^0$ in such a way that  $$g\sigma(x) = \sigma(x)$$ for all $g\in G, x \in \Omega$. 
 Set $E^1:=G\times \Omega$ and define $r, s \: E^1\to E^0$ by $$r(h,x) = h \rho(x), \quad s(h,x) = \sigma(x)$$
for all $(h,x) \in E^1$. Then one checks readily that $E=(E^1, E^0, r, s)$ is a topological graph. Moreover, letting $G$ act on $E^1$ by $g(h,x) = (gh,x)$ for all $g\in G$ and $(h,x)\in E^1$, we obtain an action of $G$ on $E$ that is easily seen to satisfy $gs(e) = s(e) $ for all $e\in E^1$.
\end{rem}

\begin{q}
If the action of $G$ on $E$ fixes sources,
what can be said about the correspondence $\yphi$ for the cocycle $\varphi(g,e)=1$?
In the case where $E$ is finite with no sources,
Exel and Pardo \cite[Example~3.6]{EP} show that $\OO_{G,E}\simeq C^*(E)$. But we 
would like
to understand this (admittedly rather trivial) situation better.
We discuss a special case toward the end of \exref{strings}.
\end{q}

 \section{Generators and relations} \label{gen-rel}

Throughout this section, $E=(E^0,E^1,r,s)$ will be a directed graph,
$G$ will be a discrete group acting on $E$,
and $\varphi\:G\times E^1\to G$ will be a cocycle for this action.
We will describe the Toeplitz algebra $\TT_{\yphi}$ in terms of generators and relations, and give a similar description of the Cuntz-Pimsner algebra $\OO_{\yphi}$ when $E$ is assumed to be row-finite. In the case where $E$ is finite and sourceless, we thereby recover Exel and Pardo's initial definition of $\OO_{G,E}$ in \cite[Section~3]{EP}.

We use the notation introduced in Section 3 and refer the reader to \cite{Kat04} for undefined terminology and notation on $C^*$-correspondences. Let $(t_{\yphi},t_B)$ denote the universal Toeplitz representation of $(\yphi,B)$ in $\ty$,
$(k_{\yphi},k_B)$ 
the universal Cuntz-Pimsner covariant representation of $(\yphi,B)$ in $\oy$,
and $(i_A,i_G)$ 
the universal covariant homomorphism of $(A,G)$ in $M(B)$.
(Of course, since $G$ is discrete we have $i_A\:A\to B$.)

We work with the crossed product $B=A\rtimes_\alpha G$ and the $B$-cor\-res\-pon\-dence $\yphi$ in terms of the generators:
\begin{itemize}
\item $\Chi_{e,g}$ denotes the element of $C_c(E^1\times G)\subset \yphi$ given by the characteristic function of $\{(e,g)\}$.
\item $\delta_{v,g}$ denotes the element of $C_c(E^0\times G)\subset B$ given by the characteristic function of $\{(v,g)\}$.
\item Similarly for $\Chi_e\in C_c(E^1)\subset X$ and $\delta_v\in C_c(E^0)\subset A$.
\end{itemize}
Thus
\begin{itemize}
\item $C_c(E^1\times G)=\spn\{\Chi_{e,g}:e\in E^1,g\in G\}$.
\item $C_c(E^0\times G)=\spn\{\delta_{v,g}:v\in E^0,g\in G\}$.
\item 
$A$ is the $c_0$-direct sum of the 1-dimensional ideals generated by the projections $\delta_v$ for $v\in E^0$.
\end{itemize}

\begin{defn}
Let $D$ be a $C^*$-algebra.
A \emph{representation of $(E,G,\varphi)$ in $D$}
is a family $\{P_v,S_e,U_g:v\in E^0,e\in E^1,g\in G\}$
such that:
\begin{enumerate}
\eplist
\item \label{ep 1}
$\{P_v,S_e:v\in E^0,e\in E^1\}$ is a Toeplitz $E$-family in $D$,

\item \label{ep 2}
the map $U:g \mapsto U_g$ is a unitary representation of $G$ in $M(D)$,

\item \label{ep 3}
for all $g\in G$, $v\in E^0$, and $e\in E^1$ we have
\[
U_gP_v=P_{gv}U_g\midtext{and}U_gS_e=S_{ge}U_{\varphi(g,e)},
\]
and

\item \label{ep 4}
$D$ is generated as a $C^*$-algebra by
\[
\{P_vU_g:v\in E^0,g\in G\}\cup \{S_eU_g:e\in E^1,g\in G\}.
\]
\end{enumerate}
We frequently shorten the notation for the family to $\{P_v,S_e,U_g\}$.
If the above condition \ep{ep 1} is replaced by
\begin{enumerate}
\epprimelist
\item \label{epprime 1}
$\{P_v,S_e:v\in E^0,e\in E^1\}$ is a Cuntz-Krieger $E$-family in $D$,
\end{enumerate}
then we say $\{P_v,S_e,U_g\}$ is a \emph{CK-representation of $(E,G,\varphi)$ in $D$}.
\end{defn}

Since the left $B$-module $\yphi$ is nondegenerate,
the canonical homomorphisms
$t_B\:B\to \ty$ and $k_B\:B\to \oy$
are nondegenerate;
we denote by $\bar{t_B}$ and $\bar{k_B}$ their extensions to the multiplier algebra $M(B)$.
We define a representation
$\{p_v,s_e,u_g\}$ of $(E,G,\varphi)$ in $\ty$ by
\begin{align*}
p_v&=t_B(i_A(\delta_v))
\\
s_e&=t_{\yphi}(\Chi_{e,1})
\\
u_g&=\bar{t_B}(i_G(g)),
\end{align*}
and a CK-representation $\{p'_v,s'_e,u'_g\}$ in $\oy$ by
\begin{align*}
p'_v&=k_B(i_A(\delta_v))
\\
s'_e&=k_{\yphi}(\Chi_{e,1})
\\
u'_g&=\bar{k_B}(i_G(g)).
\end{align*}

Note that
\begin{align*}
i_A(\delta_v)&=\delta_{v,1}
\\
i_G(g)&=\sum_{v\in E^0}\delta_{v,g},
\end{align*}
where $1$ denotes the identity element of $G$
and
the sum converges in the strict topology of $M(B)$.
Also, the technology of discrete crossed products is set up so that
\[
i_A(\delta_v)i_G(g)=\delta_{v,g},
\]
and it follows that
\[
\delta_{v',g}\delta_{v,h}
=\begin{cases}
\delta_{gv,gh}\case v'=gv\\
0\ifnot.
\end{cases}
\]

We have
\begin{align*}
i_A(\delta_v)\cdot \Chi_{e,h}&=
\begin{cases}
\Chi_{e,h}\case v=r(e)\\
0\ifnot
\end{cases}
\\
\Chi_{e,h}\cdot i_A(\delta_v)&=
\begin{cases}
\Chi_{e,h}\case s(e)=hv\\
0\ifnot
\end{cases}
\\
i_G(g)\cdot \Chi_{e,h}&=\Chi_{ge,\varphi(g,e)h}
\\
\Chi_{e,h}\cdot i_G(g)&=\Chi_{e,hg}.
\end{align*}

Consequently
\begin{align*}
\delta_{v,g}\cdot \Chi_{e,h}&=
\begin{cases}
\Chi_{ge,\varphi(g,e)h}\case v=r(ge)\\
0\ifnot
\end{cases}
\\
\Chi_{e,h}\cdot \delta_{v,g}&=
\begin{cases}
\Chi_{e,hg}\case s(e)=hv\\
0\ifnot.
\end{cases}
\end{align*}

The inner product on basis elements satisfies
\[
\<\Chi_{e,1},\Chi_{e',1}\>=
\begin{cases}
\delta_{s(e),1}\case e=e'\\
0\ifnot,
\end{cases}
\]
and so
\begin{align*}
\<\Chi_{e,g},\Chi_{e,h}\>
&=\bigl\<\Chi_{e,1}\cdot i_G(g),\Chi_{e,1}\cdot i_G(h)\bigr\>
\\&=i_G(g\inv)\<\Chi_{e,1},\Chi_{e,1}\>i_G(h)
\\&=i_G(g\inv)\delta_{s(e),1}i_G(h)
\\&=\delta_{s(g\inv e),g\inv h},
\end{align*}
while $\<\Chi_{e,g},\Chi_{e',h}\>=0$ if $e\ne e'$.

Also,
\[
\sum_{v\in E^0}\delta_v=1
\]
in $M(A)$, where the series converges strictly,
and similarly
\[
\sum_{v\in E^0}\delta_{v,1}=1
\]
in $M(B)$.
Consequently (as has been mentioned in the literature, probably many times),
\[
\sum_{v\in E^0}p_v=1
\]
strictly in $M(\TT_{\yphi})$, and similarly for the projections $p_v'$ in $M(\OO_{\yphi})$.

The following shows that $\ty$ is the universal $C^*$-algebra for representations of $(E,G,\varphi)$,
which gives a presentation of $\ty$ in terms of generators and relations.

\begin{thm}\label{universal T}
Let $\{P_v,S_e,U_g\}$ be a representation of $(E,G,\varphi)$ in a $C^*$-algebra $D$.
Then there is a unique surjective homomorphism 
$\Phi$ from $\ty$ onto $D$ 
such that
\[
\Phi(p_v)=P_v,\quad
\Phi(s_e)=S_e, \midtext{and}
\overline{\Phi}(u_g)=U_g
\]
for all $v\in E^0$, $e\in E^1$, and $g\in G$.
\end{thm}

\begin{proof}
We will construct a Toeplitz representation $(\psi,\zeta)$ of the correspondence $(\yphi,B)$ in $D$,
and we work primarily with the generators.
First of all, the family $\{P_v\}$ of orthogonal projections uniquely determines a homomorphism $\wilde P\:A\to D$.
Then the relation \ep{ep 3} immediately implies that the pair $(\wilde P,U)$ is a covariant homomorphism of the system $(A,G,\alpha)$ in $D$, and the integrated form is a homomorphism $\zeta\:B\to D$,
given on generators by
\[
\zeta(\delta_{v,g})=P_vU_g\righttext{for}v\in E^0,g\in G.
\]
Next, we define a linear map $\psi\:C_c(E^1\times G)\to D$ as the unique linear extension of the map given on generators by
\[
\psi(\Chi_{e,g})=S_eU_g\righttext{for}e\in E^1,g\in G.
\]
The computation
\begin{align*}
\psi(\Chi_{e,g})^*\psi(\Chi_{e',h})
&=(S_eU_g)^*(S_{e'}U_h)
\\&=U_g^*S_e^*S_{e'}U_h,
\intertext{which is 0 unless $e=e'$, in which case we can continue as}
&=U_{g\inv}P_{s(e)}U_h
\\&=P_{gs(e)}U_{g\inv h}
\\&=\delta_{gs(e),g\inv h}
\\&=\zeta\bigl(\<\Chi_{e,g},\Chi_{e',h}\>\bigr)
\righttext{(which is also 0 if $e\ne e'$)}
\end{align*}
implies that $\psi$ is bounded, and hence extends uniquely to a bounded linear map $\psi\:\yphi\to D$.
Then combining the above with the computation
\begin{align*}
\zeta(\delta_{v,g})\psi(\Chi_{e,h})
&=P_vU_gS_eU_h
\\&=P_vS_{ge}U_{\varphi(g,e)h},
\intertext{which is 0 unless $v=r(ge)$, in which case we can continue as}
&=S_{ge}U_{\varphi(g,e)h}
\\&=\psi(\Chi_{ge,\varphi(g,e)h})
\\&=\psi\bigl(\delta_{v,g}\cdot \Chi_{e,h}\bigr)
\righttext{(which is also 0 if $v\ne r(ge)$)}
\end{align*}
shows that $(\psi,\zeta)$ is a Toeplitz representation of $(Y^\varphi,B)$ in $D$.

The
associated homomorphism $\Phi=\psi\times_\TT \zeta$ from $\TT_{\yphi}$ to $D$ is surjective, by the properties of representations. Moreover,
by construction
this homomorphism is the unique one satisfying
\[
\Phi(p_v)=P_v,\quad
\Phi(s_e)=S_e,\midtext{and}
\overline{\Phi}(u_g)=U_g
\]
for all $v\in E^0$, $e\in E^1$, and $g\in G$.
\end{proof}

In the following lemma we will gather some information about the Katsura ideal of the $B$-correspondence $\yphi$, under a mild assumption on $E$.

\begin{lem}\label{JY}
Assume that $E$ is row-finite, i.e., $|r\inv(v)| < \infty $ for all $v\in E^0$. 
Then the Katsura ideal $J_X$ for the $A$-correspondence $X$ is $G$-invariant,
the image of the
left-module map $\phi\:B\to\LL(\yphi)$
is contained in $\KK(\yphi)$,
and
\begin{align*}
J_{\yphi}
&\subset J_X\rtimes_\alpha G
=\clspn\{\delta_{v,g}:(v,g)\in E^0\times G,0<|r\inv(v)|\}.
\end{align*}
\end{lem}

\begin{proof}
Let
\begin{align*}
E^0_{\rg}&=r(E^1)
\\
E^0_{\so}&=E^0\minus E^0_{\rg}.
\end{align*}
Thus $E^0_{\so}$ is the set of sources and $E^0_{\rg}$ is the set of regular vertices.
Also, $E^0$ is the disjoint union of these two $G$-invariant subsets.
Put $J_0=c_0(E^0_{\so})$.
It is well-known that
\begin{align*}
J_X&=c_0(E^0_{\rg})
\\
J_0&=\ker\pi.
\end{align*}
Moreover, we have a direct-sum decomposition $A=J_X\oplus J_0$ into complementary $G$-invariant ideals.
The crossed product is thus a direct sum
\[
B=(J_X\rtimes_\alpha G)\oplus (J_0\rtimes_\alpha G)
\]
of complementary ideals.
Since the left-module map $\phi\:B\to\LL(\yphi)$ coincides with $\pi \times V$,
we get
\[
J_0\rtimes_\alpha G\subset \ker\phi.
\]
Thus
\[
(\ker\phi)\ann\subset (J_0\rtimes_\alpha G)\ann=J_X\rtimes_\alpha G.
\]
As $J_{\yphi}= \phi^{-1}(\KK(\yphi)) \cap (\ker\phi)\ann$, we can finish by showing that $\phi(B)\subset \KK(\yphi)$, and by the above it suffices to show that if $v\in E^0_{\rg}$ and $g\in G$ then
\begin{equation}\label{finite rank}
\phi(\delta_{v,g})
=\sum_{r(e)=v}\theta_{\Chi_{e,1},\Chi_{g\inv e,\varphi(g\inv,e)}},
\end{equation}
which is finite rank.
Since $C_c(E^1\times G)$ is dense in $\yphi$, it suffices to check the equality of the above two operators on a basis vector $\Chi_{e',h}$.
Recall that 
$\delta_{v,g}\cdot \Chi_{e',h}=\Chi_{ge,\varphi(g,e')h}$
if $v=r(ge')$ and 0 otherwise.

We first show that
\begin{equation}\label{v,1}
\phi(\delta_{v,1})=\sum_{r(e)=v}\theta_{\Chi_{e,1},\Chi_{e,1}}.
\end{equation}
For any $e\in r\inv(v)$
we have
\[
\theta_{\Chi_{e,1},\Chi_{e,1}}\Chi_{e',h}
=\Chi_{e,1}\cdot \<\Chi_{e,1},\Chi_{e',h}\>,
\]
which is 0 if $e\ne e'$.
Thus if $r(e')\ne v$ we have
\[
\sum_{r(e)=v}\theta_{\Chi_{e,1},\Chi_{e,1}}\Chi_{e',h}
=0=\phi(\delta_{v,1})\Chi_{e',h}.
\]
So now suppose that $r(e')=v$.
Then
\begin{align*}
\sum_{r(e)=v}\theta_{\Chi_{e,1},\Chi_{e,1}}\Chi_{e',h}
&=\theta_{\Chi_{e',1},\Chi_{e',1}}\Chi_{e',h}
\\&=\Chi_{e',1}\cdot \<\Chi_{e',1},\Chi_{e,h}\>
\\&=\Chi_{e',1}\cdot \delta_{s(e'),h}
\\&=\Chi_{e',h}=\phi(\delta_{v,1})\Chi_{e',h},
\end{align*}
verifying \eqref{v,1}.

Now we can prove \eqref{finite rank}:
\begin{align*}
\phi(\delta_{v,g})
&=\phi(\delta_{v,1}i_G(g))
=\phi(\delta_{v,1})V_g
\\&=\sum_{r(e)=v}\theta_{\Chi_{e,1},\Chi_{e,1}}V_g
\\&=\sum_{r(e)=v}\theta_{\Chi_{e,1},V_g^*\Chi_{e,1}}
\\&=\sum_{r(e)=v}\theta_{\Chi_{e,1},V_{g\inv}\Chi_{e,1}}
\\&=\sum_{r(e)=v}\theta_{\Chi_{e,1},\Chi_{g\inv e,\varphi(g\inv,e)}}.
\qedhere
\end{align*}
\end{proof}

We can now deduce that,
keeping the row-finiteness
assumption on the graph $E$,
the Cuntz-Pimsner algebra
$\oy$ is the universal $C^*$-algebra for CK-representations of $(E,G,\varphi)$,
which gives a presentation of $\oy$ in terms of generators and relations.

\begin{cor}\label{universal O}
Let $\{P_v,S_e,U_g\}$ be a CK-representation of $(E,G,\varphi)$ in a $C^*$-algebra $D$.
If $E$ is row-finite,
then
there is a unique surjective homomorphism $\Phi\:\oy\to D$ such that
\[
\Phi(p'_v)=P_v,\quad
\Phi(s'_e)=S_e, \midtext{and}
\overline{\Phi}(u'_g)=U_g
\]
for all $v\in E^0$, $e\in E^1$, and $g\in G$.
\end{cor}

\begin{proof}
With the notation from the proof of \thmref{universal T}, we
must show that the Toeplitz representation $(\psi,\zeta)$ is Cuntz-Pimsner covariant.
That is, we must show that for all $b$ in the Katsura ideal $J_{\yphi}$ of $\yphi$ we have $\psi^{(1)}\circ \phi(b)=\zeta(b)$,
where $\phi\:B\to \LL(\yphi)$ is the left module homomorphism
and $\psi^{(1)}\:\KK(\yphi)\to D$ is the homomorphism associated to the Toeplitz representation $(\psi,\zeta)$.
By
\lemref{JY}, and by
linearity, density, and continuity,
it suffices to compute that for all $(v,g)\in E^0_{\rg}\times G$
\begin{align*}
\psi^{(1)}\circ \phi(\delta_{v,g})
&=\psi^{(1)}\left(\sum_{r(e)=v}\theta_{\Chi_{e,1},\Chi_{g\inv e,\varphi(g\inv,e)}}\right)
\\&=\sum_{r(e)=v}\psi(\Chi_{e,1})\psi(\Chi_{g\inv e,\varphi(g\inv,e)})^*
\\&=\sum_{r(e)=v}S_e\bigl(S_{g\inv e}U_{\varphi(g\inv,e)}\bigr)^*
\\&=\sum_{r(e)=v}S_e\bigl(U_{g\inv}S_e\bigr)^*
\\&=\sum_{r(e)=v}S_eS_e^*U_g
\\&=P_vU_g
\\&=\zeta(\delta_{v,g}).
\qedhere
\end{align*}
\end{proof}

If $E$ is finite, then $\OO_{\yphi}$ is unital, and \corref{universal O} shows that it has exactly the same universal properties as the Exel-Pardo algebra   $\OO_{G,E}$ (cf.\  \cite[Definition 3.2]{EP}).  
Hence, if  $E$ is finite with no sources, then  $\oy$ is isomorphic to  $\OO_{G,E}$. We will give another proof of this fact in \corref{alt EP alg}.

In the proof of \lemref{JY} we showed that
when $E$ is row-finite we have
$\phi(B)\subset \KK(\yphi)$, where $\phi\:B\to \LL(\yphi)$ is the left-module homomorphism.
In fact, 
assuming a bit more about $E$,
we can identify the Katsura ideal:

\begin{cor}\label{ideal}
If $E$ is row-finite and has no sources, then the Katsura ideal $J_{\yphi}$ of the $B$-correspondence $\yphi$ coincides with $B$.
\end{cor}

\begin{proof}
By the preceding, we only need to show that the left-module homomorphism $\phi\:B\to \KK(\yphi)$ is injective. 
Our new hypotheses imply that, in the notation of the proof of \lemref{JY},
$B=J_X\rtimes_\alpha G$.
Recall
the CK-representation 
\begin{align*}
p_v'&=k_B(i_A(\delta_v)),
&
s_e'&=k_{\yphi}(\Chi_{e,1}),
&
u_g'&=\bar{k_B}(i_G(g))
\end{align*}
of $(E,G,\varphi)$ in $\oy$,
and let $(\psi,\zeta)$ be the associated Toeplitz representation of the $B$-correspondence $\yphi$ in $\OO_{\yphi}$,
so that in particular
\begin{align*}
\zeta=\pi_A\times u',
\end{align*}
where $\pi_A\:A\to \OO_{\yphi}$ is determined by
\[
\pi_A(a)=\sum_{v\in E^0}a(v)p_v'\righttext{for}a\in C_c(E^0).
\]
Then clearly
$\zeta=k_B$,
the canonical homomorphism from $B$ to $\oy$.
Thus $\zeta$ is injective by \cite[Proposition~4.11]{Kat04}.
Since we have shown above that $\psi^{(1)}\circ \phi=\zeta$, it follows that $\phi$ is injective.
\end{proof}

\begin{rem} \label{varphi trivial}
We imposed the row-finite
hypothesis on the graph in \corref{universal O} because otherwise it would be problematic to get our hands on the Katsura ideal $J_{\yphi}$ of the correspondence $\yphi$.
Even when $\varphi$ is the cocycle $(g,e)\mapsto g$, so that $\yphi=X\rtimes_\gamma G$, the relationship between the two ideals $J_{X\rtimes_\gamma G}$ and $J_X\rtimes_\alpha G$ of $B=A\rtimes_\alpha G$ is murky.
There are partial results:
the two ideals coincide when
$G$ is amenable \cite[Proposition~2.7]{HaoNg},
or is discrete and has Exel's Approximation Property \cite[Theorem~5.5]{bkqr},
but it is unknown whether the two ideals coincide for arbitrary $G$.
However, when $G$ is discrete, $E$ is row-finite with no sources, and $\varphi$ is the cocycle $(g,e)\mapsto g$, Corollary \ref{ideal}  gives that $J_{X\rtimes G} = B = A\rtimes G = J_X\rtimes G$, and we can then conclude from \cite[Theorem~4.1]{bkqr} that $\OO_{X\rtimes G}\simeq \OO_X \rtimes G$, i.e., $\OO_{\yphi} \simeq C^*(E)\rtimes G$. In the case where $E$ is finite and sourceless, this was pointed out by Exel and Pardo in \cite[Example~3.5]{EP}. 
\end{rem}

\section{The Exel-Pardo correspondence}

When $G$ is discrete and the graph $E$ is finite, Exel and Pardo \cite[Section~10]{EP} define a correspondence, that they denote by $M$, over the crossed product $B=A\rtimes_\alpha G$. (Warning: they call this crossed product $A$, whereas we write $A$ for $C_0(E^0)$.)
Exel and Pardo also require $E$ to have no sources, but they remark in \cite[Section~2]{EP} that this assumption, as well as finiteness of $E$, are probably only  necessary in Section~3 of their paper, which it so happens does not concern us in our paper.

Throughout this section we assume that $G$ is a discrete group acting on a 
directed graph $E$,
and that $\varphi$ is a cocycle for this action.

Our construction of the $B$-correspondence $\yphi$ in \secref{withphi} is different from that of Exel and Pardo \cite[Section~10]{EP},
so it behooves us to compare them. 

\begin{thm}\label{M}
Let $M$ be the $B$-correspondence constructed in \cite[Section~10]{EP}. Then $\yphi\simeq M$ as $B$-correspondences.
\end{thm}

\begin{proof}
We review the construction of $M$,
but using slightly different notation and adapting it to our more general context.
It should be clear that we produce the same structure as in \cite{EP}.
For $v\in E^0$ let $\delta_v\in C_c(E^0)\subset A$ be the characteristic function of $\{v\}$.
For each $E\in E^1$ let
\[
B^e=i_A(\delta_{s(e)})B,
\]
which is a closed right ideal of $B$,
and hence a Hilbert $B$-module in the obvious way.
Then form a new Hilbert $B$-module as the direct sum
\[
M=\bigoplus_{e\in E^1}B^e.
\]
An element $m\in M$ is an $E^1$-tuple
\[
m=(m_e)_{e\in E^1},
\]
and the coordinates have the form
\[
m_e=i_A(\delta_{s(e)})b_e,\righttext{with}b_e\in B.
\]
The left $B$-module structure on $M$ is the integrated form of 
a covariant 
pair of
left module multiplications of $A$ and $G$,
defined on the generators by
\begin{align*}
(\delta_v\cdot m)_e&=\begin{cases}m_e\case v=r(e)\\0\ifnot\end{cases}
\\
(g\cdot m)_e&=i_A(\delta_{s(e)})i_G(\varphi(g,g\inv e))m_{g\inv e}.
\end{align*}

We will define an isomorphism $\Psi\:\yphi\to M$ of $B$-correspondences.
We begin by defining $\Psi$ on the dense subspace $C_c(E^1\times G)$, and by linear independence it suffices to define
\[
\bigl(\Psi\Chi_{e,g}\bigr)_{e'}=\begin{cases}
i_A(\delta_{s(e)})i_G(g)\case e'=e\\
0\ifnot.
\end{cases}
\]
The following computation implies that $\Psi$ preserves inner products on $C_c(E^1\times G)$:
for $e,f\in E^1$ and $g,h\in G$ we have
\begin{align*}
\<\Psi\Chi_{e,g},\Psi\Chi_{f,h}\>
&=\sum_{e'\in E^1}(\Psi\Chi_{e,g})_{e'}^*(\Psi\Chi_{f,h})_{e'},
\end{align*}
which is 0 unless $e=f=e'$, and when $e=f$ we have
\begin{align*}
\<\Psi\Chi_{e,g},\Psi\Chi_{e,h}\>
&=(\Psi\Chi_{e,g})_e^*(\Psi\Chi_{e,h})_e
\\&=\bigl(i_A(\delta_{s(e)})i_G(g)\bigr)^*\bigl(i_A(\delta_{s(e)})i_G(h)\bigr)
\\&=i_G(g\inv)i_A(\delta_{s(e)})i_G(h)
\\&=i_A(\delta_{g\inv s(e)}i_G(g\inv h)
\\&=\delta_{g\inv s(e),g\inv h}
\\&=\<\Chi_{e,g},\Chi_{e,h}\>.
\end{align*}
Thus $\Psi$ extends uniquely to an isometric linear map
from $\yphi$ to $M$,
which we continue to denote by $\Psi$.

As pointed out in \cite[Section~10]{EP},
\[
B^e=\clspn\{i_A(\delta_{s(e)})i_G(g):g\in G\},
\]
and it follows that $\Psi$ has dense range, and hence is surjective.

The following computations imply that $\Psi$ is right $B$-linear:
\begin{align*}
\bigl(\Psi(\Chi_{e,h}\cdot \delta_v)\bigr)_{e'}
&=\begin{cases}
(\Psi\Chi_{e,h})_{e'}\case s(e)=v\\
0\ifnot
\end{cases}
\\&=\begin{cases}
i_A(\delta_{s(e)})i_G(h)\case s(e)=v,e'=e\\
0\ifnot,
\end{cases}
\end{align*}
while
\begin{align*}
\bigl((\Psi\Chi_{e,h})\cdot \delta_v\bigr)_{e'}
&=\begin{cases}
(\Psi\Chi_{e,h})_{e'}\case s(e')=v\\
0\ifnot
\end{cases}
\\&=\begin{cases}
i_A(\delta_{s(e)})i_G(h)\case s(e')=v,e'=e\\
0\ifnot,
\end{cases}
\end{align*}
so $\Psi(\Chi_{e,h}\cdot \delta_v)=(\Psi\Chi_{e,h})\cdot \delta_v$,
and
\begin{align*}
\bigl(\Psi(\Chi_{e,h}\cdot g)\bigr)_{e'}
&=\bigl(\Psi\Chi_{e,hg}\bigr)_{e'}
\\&=\begin{cases}i_A(\delta_{s(e)})i_G(hg)\case e'=e\\0\ifnot,\end{cases}
\end{align*}
while
\begin{align*}
\bigl((\Psi\Chi_{e,h})\cdot g\bigr)_{e'}
&=(\Psi\Chi_{e,h})_{e'}\cdot g
\\&=\begin{cases}i_A(\delta_{s(e)})i_G(h)\cdot g\case e'=e\\0\ifnot\end{cases}
\\&=\begin{cases}i_A(\delta_{s(e)})i_G(hg)\case e'=e\\0\ifnot\end{cases}
\end{align*}
so $\Psi(\Chi_{e,h}\cdot g)=(\Psi\Chi_{e,h})\cdot g$.
This combined with the other properties of $\Psi$ makes it a unitary map from the Hilbert $B$-module $\yphi$ to the Hilbert $B$-module $M$
\cite[Theorem~3.5]{lance}.

The following computations imply that $\Psi$ is left $B$-linear:
\begin{align*}
\bigl(\Psi(\delta_v\cdot \Chi_{e,h})\bigr)_{e'}
&=\begin{cases}
\bigl(\Psi\Chi_{e,h}\bigr)_{e'}\case v=r(e)\\
0\ifnot
\end{cases}
\\&=\begin{cases}
i_A(\delta_{s(e)})i_G(h)\case v=r(e),e'=e\\
0\ifnot,
\end{cases}
\end{align*}
while
\begin{align*}
\bigl(\delta_v\cdot (\Psi\Chi_{e,h})\bigr)_{e'}
&=\begin{cases}
\bigl(\Psi\Chi_{e,h}\bigr)_{e'}\case v=r(e')\\
0\ifnot
\end{cases}
\\&=\begin{cases}
i_A(\delta_{s(e)})i_G(h)\case v=r(e'),e'=e\\
0\ifnot,
\end{cases}
\end{align*}
so $\Psi(\delta_v\cdot \Chi_{e,h})=\delta_v\cdot (\Psi\Chi_{e,h})$,
and
\begin{align*}
\bigl(\Psi(g\cdot \Chi_{e,h})\bigr)_{e'}
&=\bigl(\Psi\Chi_{ge,\varphi(g,e)h}\bigr)_{e'}
\\&=\begin{cases}
i_A(\delta_{s(ge)})i_G(\varphi(g,e)h)\case e'=ge\\
0\ifnot
\end{cases}
\end{align*}
while
\begin{align*}
&\bigl(g\cdot (\Psi\Chi_{e,h})\bigr)_{e'}
\\&\quad=\begin{cases}
i_A(\delta_{s(e')})i_G(\varphi(g,g\inv e'))
i_A(\delta_{s(e)})i_G(h)\case g\inv e'=e\\
0\ifnot,
\end{cases}
\intertext{since $(\Psi\Chi_{e,h})_{g\inv e'}
=i_A(\delta_{s(e)})i_G(h)$ if $g\inv e'=e$ and 0 if not,}
&=\begin{cases}
i_A(\delta_{s(ge)})i_A(\delta_{\varphi(g,e)s(e)})i_G(\varphi(g,e)h)\case e'=ge\\
0\ifnot
\end{cases}
\\&=\begin{cases}
i_A(\delta_{s(ge)})i_G(\varphi(g,e)h)\case e'=ge\\
0\ifnot,
\end{cases}
\end{align*}
since $\varphi(g,e)s(e)=gs(e)=s(ge)$,
so $\Psi(g\cdot \Chi_{e,h})=g\cdot (\Psi\Chi_{e,h})$.
Therefore $\Psi$ is an isomorphism of $B$-correspondences.
\end{proof}

\begin{cor}\label{alt EP alg}
Assume that $E$ is finite with no sources.
Then the Cuntz-Pimsner algebra $\oy$ is isomorphic to the Exel-Pardo algebra $\OO_{G,E}$.
\end{cor}

\begin{proof}
It follows immediately from \thmref{M} that the Cuntz-Pims\-ner algebras $\oy$ and $\OO_M$ are isomorphic. 
As $\OO_M$ is isomorphic to $\OO_{G,E}$ (cf.\ \cite[Theorem 10.15]{EP}), the result follows.
\end{proof}

\section{Examples}\label{examples}

\subsection{} \label{71}
Assume that a discrete group $G$ acts on a nonempty set $S$ and that
$\varphi$ is a $G$-valued cocycle for $G\act S$, so we have
\begin{equation}\label{S-cocy-eq}
\varphi(gh, x) = \varphi(g, h\cdot x)  \varphi(h, x) \righttext{for all} g, h \in G,x \in S.
\end{equation}
As in \cite[Example 3.3]{EP}, we may regard $G$ as acting on the graph $E_S$ that has one single vertex and $S$ as its edge set (so $E_S$ is bouquet of loops).
The cocycle $\varphi$ for $G\act S$ is then automatically a cocycle for $G \act E_S$, which we also denote by $\varphi$. We may then form the Toeplitz algebra $\ty$ and the Cuntz-Pimsner algebra $\oy$. Since $E_S$ is sourceless,  it follows from \corref{alt EP alg} that $\oy$ is isomorphic to the Exel-Pardo algebra $\OO_{G, E_S}$ whenever $S$ is finite. Moreover, an important motivation in \cite{EP} is that if $(G,S)$ is a self-similar group,  then $\OO_{G, E_S}$ is isomorphic to the $C^*$-algebra $\OO(G,S)$ introduced in \cite{Nek09}. Similarly,  the $C^*$-algebra $\mathcal{T}(G,S)$ studied in \cite{lrrw} is easily seen to be isomorphic to $\TT_{\yphi}$ in this case.

For completeness, we include some comments on self-similar groups (sometimes called self-similar actions) in the terminology of this paper. 
 Given an action $G\act S$ as above and a $G$-valued cocycle $\varphi$ for $G\act S$, let $S^*$ denote the set of all finite words in the alphabet $S$ and let $\varnothing \in S^*$ denote the empty word. 
One may then inductively extend the action of $G$ on $S$ to an action of $G$ on $S^*$ and $\varphi$ to a cocycle for $G\act S^*$, also denoted by $\varphi$, such that $g\cdot \varnothing = \varnothing$, $\varphi(g,\varnothing) = g$ and
\begin{equation}\label{SS1}
g\cdot(vw)=(g\cdot v)\bigl(\varphi(g,v)\cdot w\bigr),
\end{equation}
for all $g\in G$ and $v,w\in S^*$. We refer to \cite[Lemma 5.1]{Law08} for a proof. Alternatively, we note that this is just a special case of \cite[Proposition 2.4]{EP} if  one identifies $S^*$ with the set of finite paths on $E_S$.

If $S$ is finite and the action  $G \act S^*$ is faithful, then equation (\ref{SS1}) says that the pair $(G,S)$ is a self-similar group in the sense of  \cite{Nek05, Nek09} (see also \cite{lrrw}).  
(Note that $\varphi(g,v)$ is denoted by $g_{\mid v}$ in these references.) Conversely, assume that $(G,S)$ is a self-similar group, that is, $S$ is a nonempty finite set,  a faithful action of $G$ on $S^*$ fixing the empty word is given and $\varphi\:G\times S^*\to G$ is a map such that $\varphi(g,\varnothing) = g$ and (\ref{SS1}) holds. Then it can be shown (see \cite[Section 1.3]{Nek05})
that $\varphi$ is a cocycle for $G\act S^*$ and that $G\act S^*$ restricts to an action of $G$ on $S$. In particular, the restriction of $\varphi$ to $G\times S$ is a $G$-valued cocycle for  $G\act S$. 

As pointed out in \cite{Law08}, see also \cite{LW14},  it appears that self-similar (actions of) groups in a generalized sense were already considered in the 1972 thesis of Perrot, without assuming finiteness of $S$ or faithfulness of $G \act S^*$. Considering $S^*$ as the free monoid on a given set $S$, the key issue in Perrot's work is the existence of a left action $(g,w) \to g\cdot w$ of $G$ on $S^*$ such that $\varnothing$ is fixed, and of a right action $(g,w) \to \varphi(g,w)$ of $S^*$ on $G$ such that (\ref{SS1}) holds. It follows from \cite{Law08} (see in particular subsection 5.1) that this happens if and only if there exist an action of $G$ on $S$ and a $G$-valued cocycle for this action. This setting is precisely the one that is generalized in the work of Exel and Pardo.

\subsection{} A natural class of examples of $\Z$-valued cocycles for actions of $\Z$ on finite sets,  
related to the work of Katsura in \cite{Kat08} (see also \cite[Example 3.4]{EP}), is as follows. 
Let $a\in\N$ and $b\in\Z$.
For any $m\in\N$ and $k\in \Z_a$,
let $\varphi_{a,b}(m,k)\in\Z$ and $\sigma_{a,b}(m,k)\in\Z_a$ be the unique numbers satisfying
\[
bm+k=\varphi_{a,b}(m,k)a+\sigma_{a,b}(m,k).
\]
It is well-known that  
$\sigma_{a,b}\:\Z\times\Z_a\to\Z_a$ is the action of $\Z$ on  $\Z_a$ given by
\[
\sigma_{a,b}(m,k)=bm+k \mod a
\]
and that $\varphi_{a,b}\:\Z\times \Z_a\to \Z$ is a cocycle for $\sigma_{a,b}$.
We call $\varphi_{a,b}$ an \emph{EPK cocycle}
(for ``Exel-Pardo-Katsura'')
and the triple $(\Z_a,\sigma_{a,b},\varphi_{a,b})$ an \emph{EPK system}.
Clearly,  we have
\[
\sigma_{a,b+\ell a}=\sigma_{a,b}\righttext{for all}\ell\in\Z,
\]
so when $a$ is fixed we really only have $a$ distinct actions
$\sigma_{a,b}$ with $ b=0,1,\dots,a-1$. 
Moreover, for $b,b' \in\Z_a$,
the actions $\sigma_{a,b}$ and $\sigma_{a,b'}$ are conjugate if and only if
$\gcd(a,b)=\gcd(a, b')$,
so 
\[
\{\sigma_{a,d}:\text{$d\in \Z_a$ is either 0 or a positive divisor of $a$}\}
\]
forms a complete set of representatives of conjugacy classes for the actions $\sigma_{a,b}$.
However, as we will see below, something interesting happens with the cocycles.

For $b\in \Z$, writing $b = q a + r$ where $q\in \Z$ and $0 \leq r \leq a-1$,
 we get 
\[
\sigma_{a,b}(1, k) = r + k \mod a.
\]
Set $c = a-r$.
The generating function of $\varphi_{a,b} $ is then given by 
\[
\varphi_{a,b}(1,k)=\begin{cases}q\case k<c\\q+1\case k\ge c.\end{cases}
\]
Thus the signature of the cocycle $\varphi_{a,b}$ is
\[
\sum_{k=0}^{a-1}\varphi_{a,b}(1,k)=qc+(q+1)r=qa+r=b.
\]
Hence, it follows from Lemma \ref{signature} that if $b,b' \in \Z$, then the two EPK-systems $(\Z_a,\sigma_{a,b},\varphi_{a,b})$  and $(\Z_a,\sigma_{a,b'},\varphi_{a,b'})$ are  not cohomology conjugate whenever $b\neq b'$.

If $b$ is relatively prime to $a$, then the action $\sigma_{a,b}$ of $\Z$ on $\Z_a$ is obviously transitive. 
Otherwise,  the EPK-system $(\Z_a,\sigma_{a,b},\varphi_{a,b})$ may  be decomposed as follows. Setting $d= \gcd(a,b) = \gcd(a,r)$ and $a'=a/d$, one finds that there are $d$ orbits
\[
\big\{i+d\Z_a\big\}_{i=0}^{d-1},
\]
each having $a'$ elements.
Set 
 $b'=b/d$ and $r' = r/d $, 
so that $b'$ is relatively prime to $a'$ and
$b'=qa'+r'$ with $0\le r'<a'$.
For each $i=0,\dots,d-1$, the restriction of the cocycle $\varphi_{a,b}$ to the orbit $i+d\Z_a$
has generating function given by
\[
k\mapsto\begin{cases}q\case i+kd<(a'-r')d\\q+1\case i+kd\ge (a'-r')d.\end{cases}
\]
A quick computation shows that  the inequality $i+kd<(a'-r')d$ is equivalent to
$
k<(a'-r')
$ for each $i=0,\dots,d-1$. 
Thus this restricted cocycle has signature
\[
(a'-r')q+r'(q+1)=a'q+r'=b'.
\]
Since the cocycle $\varphi_{a',b'}$ for the transitive action $\sigma_{a',b'}$ of $\Z$ on $\Z_{a'}$ also has signature $b'$, we conclude from Corollary \ref{transitive} that
the restriction of the action $\sigma_{a,b}$ and the cocycle $\varphi_{a,b}$ to the orbit $i+d\Z_a$
is cohomology conjugate to the EPK system $(\Z_{a'},\sigma_{a',b'},\varphi_{a',b'})$.
(In fact, a routine computation shows that
the map
$k\mapsto i+kd$
transports the system $(\Z_{a'},\sigma_{a',b'},\varphi_{a',b'})$ to
the restriction of the action $\sigma_{a,b}$ and the cocycle $\varphi_{a,b}$ to $i+d\Z_a$.)
In this way,
we see that the EPK system
$(\Z_a, \sigma_{a,b},\varphi_{a,b})$ is cohomology conjugate with the system obtained from pasting $d$ disjoint copies of 
the transitive system
$(\Z_{a'},\sigma_{a',b'},\varphi_{a',b'})$.

More generally, let us now consider  a bijection  $\sigma$ of $\Z_a$
and let  $\xi\:  \Z_a \to \Z$. We get an action of
$\Z$ on $\Z_a$  by setting  $m\cdot k = \sigma^m(k)$ and we may then form the $\Z$-valued cocycle $\varphi$ determined by $\xi$ with respect to this action. Letting
 $E_{\Z_a}$ be the graph having one vertex and $\Z_a$ as its edge set, we get an action of $\Z$ on  $E_{\Z_a}$ and we may regard $\varphi$ as a cocycle for $\Z\act E_{\Z_a}$.   The Cuntz-Pimsner algebra $\OO_{Y^\varphi}$  is then the universal unital $C^*$-algebra generated by  Cuntz isometries $s_0, \ldots, s_{a-1}$ 
 and a unitary $u$ satisfying the relations
 \[
 us_k = s_{ \sigma(k)}u^{\xi(k)},\quad k=0, 1, \ldots, a-1.
 \]
Indeed, using these relations, one computes readily that\[
u^ms_k= s_{\sigma^m(k)} u^{\varphi(m,k)}
\]
for $m\in \Z$ and $k\in \Z_a$, and these are precisely the relations for the associated Exel-Pardo algebra.

 To ease notation, when $m \in \Z$ and $m = qa + r$ for $q\in \Z$ and $0\leq r \leq  a-1$, we will write $ q = m|a$ and $[m]_a = r $. 
If $b\in \Z$ is given, and  we let
$\sigma\:\Z_a\to \Z_a$ be 
defined by
$
\sigma(k) = [b+k]_a
$
and $\xi\:\Z_a\to \Z$ be given by
$
\xi(k) = (b+k)|a,
$
then the associated action of $\Z$ on $\Z_a$ is $\sigma_{a,b}$, while $\varphi = \varphi_{a,b}$. Hence $\OO^{a,b}:=\OO_{Y^{\varphi_{a,b}}}$  is the universal unital $C^*$-algebra generated by  Cuntz isometries $s_0, \ldots, s_{a-1}$ 
 and a unitary $u$ satisfying the relations
 \[
 us_k = s_{[b+k]_a}u^{(b+k)|a},\quad k=0, 1, \ldots, a-1.
 \]
This gives, for example, $\OO^{a, 0} \simeq \OO_a = C^*(E_{\Z_a})$ (in accordance with the fact that $\varphi_{a,0}(m,k) = 0 $ for all $m\in \Z$ and $k\in \Z_a$),
and
 $\OO^{a,a}= \OO_a\otimes C(\mathbb{T})\simeq C^*(E_{\Z_a})\rtimes_{\rm id}\Z$ (in accordance with the fact that $\varphi_{a,a}(m,k) = m $ for all $m\in \Z$ and $k\in \Z_a$).   
More interestingly,  $\OO^{2,1}$ is the universal unital $C^*$-algebra generated by two Cuntz isometries $s_0, s_1$
and a unitary $u$ satisfying the relations
\[
us_0 = s_1,\quad us_1 = s_0u.
\]
It is then not difficult to see that $\OO^{2,1}$ is the universal unital $C^*$-algebra generated by an isometry $s_0$
and a unitary $u$ satisfying the relations
\[
u^2s_0 = s_0u,\quad s_0s_0^* + us_0s_0^*u^* = 1,
\]
that is, $\OO^{2,1} \simeq \mathcal{Q}_2$, where $\mathcal{Q}_2$ is the $C^*$-algebra studied in
\cite{LarLi2adic}
(see also references therein).
As mentioned in \cite{LarLi2adic} (right after Remark 3.2), 
$\mathcal{Q}_2$ is 
isomorphic to the $C^*$-algebra $\OO(E_{2,1})$ considered in 
\cite[Example~A.6]{Ka4}.
In fact, we have $\OO^{a,b}\simeq \OO(E_{a,b})$ in general, where $E_{a,b}$ denotes the topological graph defined in \cite[Example~A.6]{Ka4}; this follows readily from the description of $\OO(E_{a,b})$ given on page 1182 of \cite{Ka4}.

\subsection{}The class of EPK-systems may be put in a general framework. Let us first remark that if a discrete group $G$ acts on a set $S \neq \varnothing$, 
$\varphi$ is a $G$-valued cocycle for $G\act S$, and $\tau$ is an endomorphism of $G$, then we may define another action $\cdot'$ of $G$ on $S$ by setting $$g\cdot'x = \tau(g) \cdot x$$ and a $G$-valued cocycle $\varphi_\tau$ for this action by setting $$\varphi_\tau(g,x)= \varphi(\tau(g), x),$$ as is easily verified. 

Next, let $\rho$ be an injective endomorphism of a discrete group $G$ and set $H=\rho(G)$. To be interesting for what follows, $G$ should be infinite and $\rho$ should not be surjective. Choose a set $S_\rho$ of coset representatives for $G/H$ containing $e$. For each $g\in G$, let $s(g)$ denote the unique element of $S_\rho$ satisfying $s(g)H = gH$. For $g\in G$ and $x\in S_\rho$,  set 
\begin{gather*}
g\cdot x = s(gx),
\\
\varphi(g,x) = \rho^{-1}(s(gx)^{-1} gx).
\end{gather*}
Since $s(gx)H = gxH$, we have $s(gx)^{-1} gx \in H=\rho(G)$, so $\varphi(g,x)$ is well-defined and lies in $G$.
It is then not difficult to check that this gives an action of $G$ on $S_\rho$,  that $\varphi$ is a cocycle for this action and that this construction does not depend on the choice of coset representatives for $G/H$, up to cohomology conjugacy.
Such a construction appears in \cite[Example 2.2]{lrrw} in the case where $G=\Z^n$ for some $n\in \N$ and $\rho\:\Z^n \to \Z^n$ is of the form $\rho(m) = Am$  for some $A\in M_n(\Z)$ with $|\det A\, | > 1$, in which case $S_\rho$ is finite with $|S_\rho| = |\det A \,| $.  

Now, let $\tau$ be another  endomorphism of $G$. We then get an action  $\cdot'$ of $G$ on $S_\rho$ and a $G$-valued cocycle $\varphi_\tau$ for this action, given by
\begin{gather*}
g\cdot'x  = s\big(\tau(g)x\big),
\\
\varphi_\tau(g,x) =  \rho^{-1}\Big(s\big(\tau(g)x\big)^{-1} \tau(g)x\Big)
\end{gather*}
for $g\in G$ and  $x\in S_\rho$.    

For example, let $G=\Z$, $a\in \N$ ($a\geq 2$) and $b\in \Z$,  set $\rho(m) = am$ and $ \tau(m)=bm$ for $m\in \Z$,  and choose $S_\rho=\Z_a$. Then the action $\cdot'$ of $\Z$ on $\Z_a$ is equal to $\sigma_{a,b}$ and  $\varphi_\tau$ is equal to $\varphi_{a,b}$, so we recover the EPK-system associated with $a$ and $b$.  When $G=\Z^n$, one may similarly consider $\rho$ associated with some $A \in M_n(\Z)$ ($|\det A\,| > 1$) and $\tau$ associated with some $B\in M_n(\Z)$. 

\subsection{}

Triples $(E, G, \varphi)$ where $G$ is a discrete group acting on a directed graph $E$, in the trivial way on $E^0$, might be produced as follows:
\begin{itemize} \item Pick a directed graph $E$ and a discrete  group $G$.
\item Let $G$ act  trivially on $E^0$.
\item For each $v,w\in E^0$, set ${}_{v}E^1_w=\{ e\in E^1: r(e) = v, s(e) = w\}$. Note that 
$E^1$ is the disjoint union of all these sets. 
\item Set $R_E= \{ (v,w) \in E^0\times E^0: {}_{v}E^1_w \neq \varnothing\}$. 
\item For each $(v,w) \in R_E$, pick an action of $G$ on ${}_{v}E^1_w$ and a cocycle ${}_{v}\varphi_w$ for this action. 
\item Paste these actions and these cocycles together to obtain an action of $G$ on $E^1$ and a cocycle $\varphi$ for it. 
\end{itemize}
Since $G$ acts  trivially on $E^0$, it is clear that we get an action of $G$ on the graph $E$ and that $\varphi$ is a cocycle for this action. Moreover, it is easy to see that 
if ${}_{v}\varphi'_w$ is  also a cocycle for the chosen action of $G$ on ${}_{v}E^1_w$ for each $(v,w) \in R_E$, then the resulting cocycle $\varphi'$ will be cohomologous to $\varphi$ if and only if ${}_{v}\varphi'_w$ is cohomologous to ${}_{v}\varphi_w$ for each $(v,w) \in R_E$.

To illustrate this procedure, set $G=\Z$ and let $E$ be a  directed graph
such
that the number 
$A(v, w)$ of edges 
in ${}_{v}E^1_w$ is finite for all $v,w \in E^0$.
Note that this hypothesis is much weaker than requiring that $E$ be \emph{locally finite} in the sense that each vertex only receives and emits finitely many edges. Let $B\:E^0\times E^0\to \Z$ be a map. For each $(v,w) \in E^0\times E^0$ such that $A(v,w) \geq 1$, i.e., for each $(v,w)\in R_E$,  we may choose a bijection from $\Z_{A(v,w)}$ onto ${}_{v}E^1_w$ and use it to transfer the EPK-system associated with the pair $A(v,w), B(v,w)$ into an action  of $\Z$ on ${}_{v}E^1_w$ and a cocycle for this action. Using these choices in the construction outlined above, we obtain an action of $\Z$ on $E$ fixing all vertices and a cocycle
$\varphi_B$ for this action. Let $B_E\:R_E\to\Z$ denote the restriction of $B$ to $R_E$. Note that if $C\:E^0 \times E^0\to \Z$ is any other map such that $B_E \neq C_E$, then it follows from our previous analysis of EPK-systems that the systems $(E,\Z,\varphi_B)$ and $(E,\Z,\varphi_{C})$ are not cohomology conjugate. Note also that if $E$ is a 
countable row-finite
graph with no sources, then we just get the class of $C^*$-algebras $\mathcal{O}_{A,B}$ introduced by Katsura in \cite{Kat08},  as presented in \cite[Example~3.4]{EP} 
when $E$ is finite with no sources. 

As a concrete example, let $a \in \N$ and consider the graph $E$ given by $E^0=\Z$, $E^1 = \Z_a \times \Z$, $r(t,j) = j-1$, and $ s(t,j) = j $ for $(t,j) \in E^1$, so that $A(i,j) = a$ when $i= j-1$ and is zero otherwise. Only the coefficients  $B(j-1,j)$ along the first subdiagonal of $B$ will then matter. 
In this example,  $E$ is row-finite with no sources, so it will give one of Katsura's $\mathcal{O}_{A,B}$. But it can easily be changed so that $E$ is not row-finite with no sources (for example by adding one edge $e_j$ (or more) going from 0 to $j$ for each $j\in \Z$), but still satisfies the requirement that $|A(i,j)| < \infty$  for all $i, j\in \Z=E^0$).

\subsection{}\label{strings}
Consider again a triple $(S, G, \varphi)$ where a discrete group $G$ acts on a set $S$ and $\varphi$ is a cocycle for this action. Pick any symbol $\omega \not\in S$. Let then $F=F_S$ be the directed graph where  $F^0 = S \cup \{\omega\}$, $F^1= S$, and $ r, s\:F^1\to F^0$ are given by
\[
r(x) = x,\quad s(x) = \omega\righttext{for} x \in F^1=S.
\]
Obviously, $F$ has exactly one source, namely $\omega$. (If $S$ is finite, $F$ may be thought of as a bouquet of $|S|$ disjoint strings (that are not loops) emanating from $\omega$.) The action of $G$ on $S$ induces a natural action of $G$ on $F$ in an obvious way: we just set  $g\omega = \omega$ for all $g\in G$,  and let $G$ act on $F^0\setminus \{\omega\} = S$ and on $F^1=S$ via its given action on $S$.  The cocycle $\varphi$ is then a cocycle for the action of $G$ on $F$: the first condition is automatically satisfied (since $F^1=S$);
because
\[
\varphi(g,x) s(x) = \varphi(g,x) \omega = \omega =g\omega = g s(x)
\]
for all $g\in G$ and $x\in F^1=S$,
the second condition is trivially satisfied.

\subsubsection*{Special case}

Set $S=G$ and let $G$ act on itself by left translation.
As the map $\id\: F^1=G \to G$ 
trivially satisfies condition (2) in 
\propref{g cby} (and the assumption in this proposition is fulfilled), we get that the cocycle $(g,e) \to g$ is a coboundary, i.e., it is cohomologous to the cocycle $(g,e)\to 1$. Hence we conclude that the correspondences associated to these cocycles are isomorphic.
For the first of these cocycles, it follows from \remref{varphi trivial} that
we have
\[
\oy = \OO_{X_F\rtimes G}\simeq C^*(F)\rtimes G,
\]
which is frequently not isomorphic to $C^*(F)$.
As an explicit example,
consider the cocycle $\varphi(g,d)=g$ for the action $\Z_2\act \Z_2$ by translation.
Since 
any action of $\Z_2$ on $C^*(F)=M_3$ is inner, we get
\begin{align*}
\OO_{\yphi} &
\simeq C^*(F)\rtimes \Z_2
\simeq M_3\rtimes \Z_2\\
&\simeq M_3\otimes \C^2
\simeq M_3\oplus M_3 \\ 
&\not\simeq M_3,
\end{align*}
and we obtain 
the same $C^*$-algebra for the cocycle $\varphi=1$.

This is in contrast to the situation in
\cite[Example~3.6]{EP},
where the graph $E$ is finite and has no sources,
and the action 
fixes the vertices; 
Exel and Pardo then show that for
the cocycle $\varphi=1$ we have
$\oy\simeq C^*(E)$,
because the unitaries $u_g$ for $g\in G$ can be expressed in terms of the partial isometries $s_e$ for $e\in E^1$.
Note that Exel and Pardo's observation does not apply to the graph $F$ above simply because $F$ has a source, namely $\omega$.

\subsection{}
A more general construction in the same vein as the one in \ref{strings} is as follows. Let $(S, G, \varphi)$ be as in \ref{strings}. Assume that we are also given an action of $G$ on a nonempty set $I$ and a $G$-equivariant map $\rho\:S\to I$. Pick a symbol $\omega \not\in I$ and let $F$ be the directed graph where $F^0 = I \cup \{ \omega\}$, $F^1=S$ and $ r, s\:F^1\to F^0$  are given by 
\[
r(x) = \rho(x),\quad s(x) = \omega\righttext{for} x \in F^1=S.
\]
The two actions of $G$ induce a natural action of $G$ on $F$ by setting  $g\omega = \omega$ for all $g\in G$ and letting  $G$ act on $F^0\setminus \{\omega\} = I$ and on $F^1=S$ via the given actions of $G$ on $I$ and $S$, respectively. The cocycle $\varphi$ is then again a cocycle for the action of $G$ on $F$.

In this example, all edges of $F$ have the same source $\omega$, which is a source for $F$, and all vertices different from $\omega$ are sinks for $F$, undoubtedly a somewhat special situation.
Next we define a similar class of examples, but without sinks.

Assume that $G$ also acts on a nonempty set $T$ and pick  a symbol $\omega \not\in S\cup T$. 
 Let  then $K = (K^0, K^1, r, s)$ be the directed graph where
 \[
 K^0 =   S \cup \{\omega\},\quad K^1= S\times (T \cup\{\omega\})
 \]
 and $ r, s\:K^1\to K^0$  are given by
 \begin{gather*}
 r(x, \omega) = x,\quad s(x,\omega) = \omega,
 \\
 r(x, y) = x = s(x,y)
 \end{gather*}
for $x \in S$ and $ y \in T.$ 
Define an action of  $G$ on $K$ as follows: 
\begin{itemize}
\item $G$ acts on $K^0\setminus \{\omega\} = S$ via the given action of $G$ on $S$,
\item $g\omega = \omega,$ 
\item $ g(x, \omega) = (gx, \omega)$ 
\item $ g(x,y) = (gx, gy)$ 
\end{itemize}

for $g \in G$, $x\in S$, and  $y \in T$.
Moreover, define $\widetilde{\varphi}\: G\times K^1\to G$ by
\begin{align*}
\widetilde{\varphi}\big(g, (x,\omega)\big) &= \varphi(g, x),
\\
\widetilde{\varphi}\big(g, (x,y)\big)&= g
\end{align*}
for $g \in G$, $x\in S$, and $ y \in T$. 
Then $\widetilde{\varphi}$ is a cocycle for $G \act K$ that is not cohomologous to the trivial cocycle if $\varphi$ is not  cohomologous to the trivial cocycle for $G \act S$.

This graph  still has one source, namely $\omega$. To obtain a system with a graph having no sources, one can for example add one loop  (or more) at $\omega$, let $G$ act on this loop (or these loops) by fixing it (or them), and set $\varphi(g, e) = g$ for all $g$ when $e$ is this loop (or any of these loops).

\subsection{}
In \cite[Example~2]{Ka2}, Katsura constructs a topological graph
from a locally compact Hausdorff space $S$
and 
a homeomorphism $\sigma\:S\to S$.
The associated topological graph $E_\sigma$ has
$E_\sigma^0=E_\sigma^1=S$, 
$s=\id_S$, and $r=\sigma$.
The main point of this class of examples of topological graphs is the natural isomorphism
\[
C^*(E_\sigma)\simeq C_0(S)\rtimes_\alpha \Z,
\]
where $\alpha$ is the associated action of $\Z$ on $C_0(S)$.

Actions of $\Z$ on the topological graph $E_\sigma$ are in 1-1 correspondence with homeomorphisms $\tau\:S\to S$ that commute with $\sigma$,
via
$n\cdot x=\tau^n(x)$
for $n\in\Z,x\in S$.
We can regard a cocycle $\varphi$ for such an action as a continuous map $\varphi\:\Z\times S\to \Z$,
and the generating function of $\varphi$ as a continuous map $\xi\:S\to \Z$ satisfying
\[
\xi(x)-1\in S_x:=\{k\in\Z:\tau^k(x)=x\},
\]
so that $\xi(x)$ is congruent to 1 modulo the period of the orbit $\Z\cdot x$ (where by convention the period is defined to be 0 if the orbit is free, in which case $\varphi(n,x)=n$ for all $n\in\Z$).

\subsection{}

Assume that  $H$ is a discrete group acting by homeomorphisms on a locally compact space Hausdorff  space $E^0$. Set $E^1= H\times E^0$ and define $r,s\:E^1\to E^0$ by 
$$r(h,x) = h\cdot x, \quad s(h,x) = x$$
for all $(h,x) \in E^1$. This gives a topological graph $E$. Note that  $C^*(E)$ is in general not isomorphic to  $C_0(E^0) \rtimes H$. (For example, if $H$ is finite and abelian, $E^0$ is finite and the action of $H$ on $E^0$ is trivial, then $C_0(E^0) \rtimes H\simeq C(E^0) \otimes C^*(H)$ is abelian, while $C^*(E)$ is the direct sum of $|E^0| $ copies of the Cuntz algebra $\OO_{|H|}$). 

Now, assume that a discrete group $G$ also acts on $E^0$ by homeomorphisms and that this action commutes with the action of $H$. We may then define an action of $G$ on $E^1$ by 
$$ g\cdot(h,x) = (h, g\cdot x)$$
One easily verifies that this gives an action of $G$ on $E$. 

Let $\phi\:G\times E^0 \to G$ be a cocycle for $G\curvearrowright E^0$ satisfying
$$\phi(g,x)\cdot x = g\cdot x$$
for all $(g,x)\in G\times E^0$. Then the map $\varphi\: G\times E^1\to G$ defined by $$\varphi\big(g,(h,x)\big) = \phi(g, h\cdot x)$$
is a cocycle for the action of $G$ on $E$. 
Indeed, since the actions of $G$ and $H$ on $E^0$ commute, we have
\begin{align*} 
\varphi\big(g_1g_2, (h,x)\big) &= \phi(g_1g_2, h\cdot x) \\
&= \phi\big(g_1, g_2\cdot(h\cdot x)\big) \phi(g_2, h\cdot x)\\
&= \phi\big(g_1, h\cdot(g_2\cdot x)\big) \phi(g_2, h\cdot x)\\
&= \varphi\big(g_1, (h, g_2\cdot x)\big) \varphi\big(g_2, (h,x)\big)\\
&=\varphi\big(g_1, g_2\cdot(h,x)\big) \varphi\big(g_2, (h,x)\big)
\end{align*}
for all $g_1, g_2 \in G, (h,x)\in E^1$, and
\begin{align*} 
\varphi\big(g, (h,x)\big)\cdot s(h,x)&= \phi(g, h\cdot x)\cdot x \\
&= h^{-1}\cdot\big(\phi(g, h\cdot x)\cdot (h\cdot x)\big)\\
&= h^{-1}\cdot\big(g\cdot(h\cdot x)\big) = g\cdot x\\
&= g\cdot s(h,x)
\end{align*} 
for all $g\in G, (h,x)\in E^1$.

Note that if the action of $G$ on $E^0$ is free, then $\phi$ has to be the trivial cocycle $(g, x) \to g$ for the action $G\act E^0$, so $\varphi$ can only be the trivial cocycle for $G\act E$. A simple example where the action $G\act E^0$ is not free is as follows. Set $E^0 = \T$, $G=H=\Z$, pick  $\lambda, \mu \in \T$ such that $\lambda$ has period $p$, and $k\in \Z$. Let $G\act E^0$ (resp.\ $H \act E^0$) be given by $(m,z) \to \lambda^m z$ (resp.\ $(n,z) \to \mu^n z$) and define $\phi: G\times E^0\to G$ by $\phi(m, z) = (1 + kp)m$. Then $G\act E^0$ is not free and all the required conditions are easily verified. Note that the cocycle $\varphi$ we get for the action of $G$ on $E$ is simply given by
$\varphi(m, (n,z)) = (1 + kp)m$. It would be interesting to know whether more exotic examples can be produced.  

\subsection{}
Assume that a discrete group $G$ acts on a nonempty set $S$ and that
$\varphi$ is a $G$-valued cocycle for $G\act S$. We recall from \subsecref{71} that  $G \act S$ extends to an action of $G$ on $S^*$, where $S^*$ denotes the set of words on the alphabet $S$, and that $\varphi$ extends to a cocycle for $G\act S^*$, also denoted by $\varphi$.

As Nekrashevych \cite{Nek09} points out in the case of a self-similar group, see also  \cite[Section~2]{lrrw}, $S^*$ may be used to build a directed rooted tree $T$ (sometimes called an arborescence),
with the empty word
$\varnothing$
as the root,
and with vertex set $T^0=S^*$
and edge set
\[
T^1=\{(w,wx):w\in S^*,x\in S\}.
\]
In view of our conventions (which in this respect conform to those of \cite{EP}),
namely that paths in a directed graph should go from right to left,
we dictate that an edge $(w,wx)$ has source $wx$ and range $w$. Since $G$ acts on $S^*=T^0$, we clearly get an action of $G$ on $T$ when we define $G\act T^1$ by setting $$g\cdot(w,wx)= \big(g\cdot w, g\cdot(wx)\big)$$
 for $g\in G, w\in S^*$ and $x\in S$. We can also define a map $\varphi\:G\times T^1 \to G$ by $$\varphi\big(g, (w,wx)\big) = \varphi(g, wx)$$
  for $g\in G, w\in S^*$ and $x\in S$. It is then straightforward to check that $\varphi$ is a cocycle for $G\act T^1$. To become a graph cocycle for $G\act T$, $\varphi$ must satisfy 
  \begin{equation} \label{CT1} 
   \varphi\big(g, (w, wx)\big)\cdot (wx) = g\cdot (wx) \quad \text{ for all } g\in G, w\in S^*, 
   x\in S,
     \end{equation}
  that is,
  \begin{equation} \label{CT2} 
   \varphi(g,wx)\cdot (wx) = g\cdot (wx) \quad \text{ for all } g\in G, w\in S^*, 
   x\in S.
  \end{equation}
In particular, $\varphi$ must then satisfy 
  \begin{equation} \label{CT3}
  \varphi(g,x)\cdot x = g\cdot x \quad    \text{ for all } g\in G \text{ and } x \in S.
  \end{equation}
If $G\act S$ is free, then \eqref{CT3} only holds when $\varphi$  is the trivial cocycle $(g,x) \mapsto g$ for $G\act S$, hence the trivial cocycle is the only possible one for $G\act T$. 
Interestingly, this is also the case if we assume that the action $G\act S^*$ is faithful and, instead of \eqref{CT1}, we impose the stronger Exel-Pardo vertex condition $\varphi(g,e)\cdot v=g\cdot v$ for all $g\in G, e\in T^1$, and $v\in T^0$.
It is not difficult to construct examples where $G\act S$ is not free and there exist cocycles $\varphi$ for $G \act S$ that satisfy \eqref{CT3} and are different from the trivial cocycle. 
Conceivably, there might exist cases
 where such cocycles  satisfy \eqref{CT2}, that is, give cocycles for $G\act T$, 
 but we don't know of any concrete example.



\begin{thebibliography}{12}

\bibitem{AD87}
C.~Anantharaman-Delaroche, \emph{Syst{\`e}mes dynamiques non commutatifs et moyennabilit{\'e}}, Math. Ann., \textbf{279} (1987), 297--315.

\bibitem{AD02}
\bysame, \emph{Amenability and exactness for dynamical
  systems and their {$C^\ast$}-algebras}, Trans. Amer. Math. Soc. \textbf{354}
  (2002), no.~10, 4153--4178 (electronic).

\bibitem{bkqr}
E.~B\'edos, S.~Kaliszewski, J.~Quigg, and D.~Robertson, \emph{{A new look at
  crossed product correspondences and associated $C^*$-algebras}}, J. Math.
  Anal. Appl. \textbf{426} (2015), no.~2, 1080--1098.

\bibitem{dkq}
V.~Deaconu, A.~Kumjian, and J.~Quigg, \emph{Group actions on topological
  graphs}, Ergodic Theory Dynam. Systems \textbf{32} (2012), 1527--1566.

\bibitem{taco}
S.~Echterhoff, S.~Kaliszewski, J.~Quigg, and I.~Raeburn, \emph{{Naturality and
  induced representations}}, Bull. Austral. Math. Soc. \textbf{61} (2000),
  415--438.
  
\bibitem{enchilada}
\bysame, \emph{{A categorical approach to imprimitivity theorems for {$C^*$}-dynamical systems}}, vol.\ 180, Mem. Amer. Math. Soc., no.\ 850, American Mathematical Society, Providence, RI, 2006.

\bibitem{EP}
R.~Exel and E.~Pardo, \emph{{Self-similar graphs, a unified treatment of
  Katsura and Nekrashevych $C^*$-algebras}}, Adv. in Math. \textbf{306} (2017), 1046--1129.

\bibitem{HaoNg}
G.~Hao and C.-K. Ng, \emph{Crossed products of {$C^*$}-correspondences by
  amenable group actions}, J. Math. Anal. Appl. \textbf{345} (2008), no.~2,
  702--707.

\bibitem{Ka1}
T.~Katsura, \emph{A class of {$C\sp \ast$}-algebras generalizing both graph
  algebras and homeomorphism {$C\sp \ast$}-algebras. {I}. {F}undamental
  results}, Trans. Amer. Math. Soc. \textbf{356} (2004), no.~11, 4287--4322.

\bibitem{Kat04}
\bysame, \emph{On {$C^*$}-algebras associated with {$C^*$}-correspondences}, J.
  Funct. Anal. \textbf{217} (2004), no.~2, 366--401.

\bibitem{Ka2}
\bysame, \emph{A class of {$C^*$}-algebras generalizing both graph algebras and
  homeomorphism {$C^*$}-algebras. {II}. {E}xamples}, Internat. J. Math.
  \textbf{17} (2006), no.~7, 791--833.

\bibitem{Ka4}
\bysame, \emph{A class of {$C^*$}-algebras generalizing both graph algebras and
  homeomorphism {$C^*$}-algebras. {IV}. {P}ure infiniteness}, J. Funct. Anal.
  \textbf{254} (2008), no.~5, 1161--1187.

\bibitem{Kat08}
\bysame, \emph{A construction of actions on {K}irchberg algebras which induce
  given actions on their {$K$}-groups}, J. Reine Angew. Math. \textbf{617}
  (2008), 27--65.

\bibitem{lrrw}
M.~Laca, I.~Raeburn, J.~Ramagge, and M.~F. Whittaker, \emph{Equilibrium states
  on the {C}untz-{P}imsner algebras of self-similar actions}, J. Funct. Anal.
  \textbf{266} (2014), no.~11, 6619--6661.

\bibitem{lance}
E.~C. Lance, \emph{{Hilbert $C^*$-modules}}, London Math. Soc. Lecture Note
  Ser., vol. 210, Cambridge University Press, 1995.

\bibitem{LarLi2adic}
N.~S. Larsen and X.~Li, \emph{The 2-adic ring {$C^\ast$}-algebra of the
  integers and its representations}, J. Funct. Anal. \textbf{262} (2012),
  no.~4, 1392--1426.

\bibitem{Law08} M.~V.~Lawson, \emph{A correspondence between a class of monoids and self-similar group actions I},
Semigroup Forum \textbf{76} (2008), 489--517.

\bibitem{LW14} M.~V.~Lawson and A.~R.~Wallis, \emph{A correspondence between a class of monoids and self-similar group actions II}, Internat. J. Algebra Comput. \textbf{25} (2015), no.~4, 644--668.

\bibitem{Nek05}
V.~Nekrashevych, \emph{Self-similar groups}, Mathematical Surveys and
  Monographs, vol. 117, American Mathematical Society, Providence, RI, 2005.

\bibitem{Nek09}
\bysame, \emph{{$C^*$-algebras and self-similar groups}}, J. Reine
  Angew. Math. \textbf{630} (2009), 59--123.

\bibitem{Rae} I.~Raeburn, \emph{Graph Algebras}, CBMS Reg. Conf. Ser. Math., vol. 103, American Mathematical Society, Providence, RI, 2005, published for the Conference Board of the Mathematical Sciences, Washington, DC.

\bibitem{Spi11}
J.~Spielberg, \emph{Groupoids and {$C^*$}-algebras for categories of paths},
  Trans. Amer. Math. Soc. \textbf{366} (2014), no.~11, 5771--5819.

\bibitem{danacrossed}
D.~P. Williams, \emph{Crossed products of {$C{\sp \ast}$}-algebras},
  Mathematical Surveys and Monographs, vol. 134, American Mathematical Society,
  Providence, RI, 2007.

\bibitem{Zimmer}
R.~J. Zimmer, \emph{Ergodic theory and semisimple groups}, Monographs in
  Mathematics, vol.~81, Birkh\"auser Verlag, Basel, 1984.


\end{thebibliography}

\providecommand{\bysame}{\leavevmode\hbox to3em{\hrulefill}\thinspace}
\providecommand{\MR}{\relax\ifhmode\unskip\space\fi MR }
\providecommand{\MRhref}[2]{%
  \href{http://www.ams.org/mathscinet-getitem?mr=#1}{#2}
}
\providecommand{\href}[2]{#2}

\end{document}